\documentclass{article}
\usepackage[utf8]{inputenc}
\usepackage{bbm}
\usepackage{amsfonts}
\usepackage{latexsym, amssymb, amsmath, amscd, amsthm, amsxtra}
\usepackage{mathtools}
\usepackage{enumerate}
\usepackage{color}
\usepackage{comment}
\usepackage{tikz}
\usetikzlibrary{external}
\usepackage[compat=1.1.0]{tikz-feynman}

\usepackage{amsthm}

\title{generalized wigner}
\date{\today}

\newtheorem{theorem}{Theorem}[section]
\newtheorem{definition}[theorem]{Definition}
\newtheorem{assumption}[theorem]{Assumption}
\newtheorem{lemma}[theorem]{Lemma}
\newtheorem{remark}[theorem]{Remark}
\newtheorem{corollary}[theorem]{Corollary}

\newcommand{\ti}{\text{i}}
\newcommand{\dg}{\mathrm{diag}}
\newcommand{\E}{\mathbb E}
\newcommand{\N}{\mathbb N}
\newcommand{\R}{\mathbb R}
\newcommand{\Ex}[1]{\left\langle {#1}\right\rangle}
\newcommand{\Val}{\mathrm{Val}}
\newcommand{\tr}{\mathrm{tr}}
\newcommand{\ext}{\mathrm{ext}}
\usepackage{xcolor}

\numberwithin{equation}{section}

\usepackage{hyperref}
\hypersetup{
  pdfcreator = {},
  pdfproducer = {}
}

\usepackage[ margin=1.3 in]{geometry}

\begin{document}
\title{Eigenstate Thermalization Hypothesis for Generalized Wigner Matrices}

\author{Arka Adhikari\thanks{Department of Mathematics, Stanford University, Stanford CA 94305-2125, USA. Supported in part by NSF grant DMS-2102842}
\and Sofiia Dubova \thanks{Department of Mathematics, Harvard University, Cambridge MA 02138, USA.}
\and Changji Xu\thanks{Center of Mathematical Sciences and Applications, Harvard University, Cambridge MA 02138,USA.}
\and Jun Yin\thanks{Department of Mathematics,  University of California, Los Angeles, Los Angeles, CA 90095, USA. Supported in
part by the NSF grant DMS-1802861.}}

\maketitle
\begin{abstract}
    In this paper, we extend results of Eigenvector Thermalization to the case of generalized Wigner matrices. Analytically, the central quantity of interest here are multiresolvent traces, such as $\Lambda_A:= \frac{1}{N} \text{Tr }{ GAGA}$. In the case of Wigner matrices, as in \cite{cipolloni-erdos-schroder-2021}, one can form a self-consistent equation for a single $\Lambda_A$. There are multiple difficulties extending this logic to the case of general covariances. The correlation structure does not naturally lead to deriving a closed equation for $\Lambda_A$; this is due to the introduction of new terms that are quite distinct from the form of $\Lambda_A$. We find a way around this by carefully splitting these new terms and writing them as sums of $\Lambda_B$, for matrices $B$ obtained by modifying $A$ using the covariance matrix. The result is a system of inequalities relating families of deterministic matrices. Our main effort in this work is to derive  this system of  inequalities. 
    
\end{abstract}
\section{Introduction}
\subsection{Background and History}

Ever since Wigner proposed the study of random matrices in 1960 \cite{Wigner} in order to understand the energy spectra of heavy atoms, there has been significant effort in trying to understand the behavior of the eigenvalue fluctuations of random matrices. Wigner's celebrated conjecture states that the statistics of the eigenvalue differences should only depend on the symmetry class of the model, not on the details of the randomness that generated the model. There have been multiple works in recent years that shed light on this phenomenon, \cite{ErdSchYau2011, TaoVu2011}.

Even though the eigenvalues of random matrices are relatively well understood, the eigenvector statistics of random matrices remain largely mysterious. In contrast to the statistics of eigenvalue distributions where there are many powerful tools such as the Dyson-Brownian motion \cite{ErdSchYau2011,ErdYau2012,ErdYauYin2012Rig,ErdPecRamSchYau2010,ErdYau2012singlegap,BouErdYauYin2015,BouErdYauYin2017,ErdYauYin2012Univ,KnoYin2013,ErdKruSchro20,ErdKruScho19}, the four moment method \cite{TaoVu2011}, and direct computation via the Brezin-Hikami formula \cite{BreHik96, BreHik97}, the equations determining the behavior of the eigenvector are less amenable to analysis.   

{ There are multiple conjectures regarding the behavior of the eigenvectors of random matrices inspired by conjectures derived from studying the quantum analogues of dynamical systems. The BGS conjecture \cite{BGS} proposed that the eigenvalue behavior of the quantum analogues of classically chaotic dynamical systems should follow appropriate random matrix statistics; this conjecture, and various others, suggested a deep link between dynamical systems and random matrix theory. The study of eigenstates of these quantum dynamical operators led to very rich behavior; such as the celebrated Quantum Unique Ergodicity conjecture by Rudnik and Sarnak \cite{RudSar1994}. This suggests that, as $i \to \infty$, all eigenstates $\phi_i(x)$ of the Laplace-Beltrami operator on a surface with ergodic geodesic flow  have an associated measure $|\phi_i(x)|^2 \text{d}x$ that becomes flat as $i \to \infty$, except for an exceptional sequence. The Eigenstate Thermalization Hypothesis \cite{deutsch91,deutsch18, srednicki94} is implied by similar claims regarding the value of the related observable $\langle \phi_i, A \phi_j \rangle$ $i,j \to \infty$,  for appropriate operators $A$. For further discussion of these results and references, refer to \cite{cipolloni-erdos-schroder-2021}.

There has recently been significant work in the random matrix community, to try to find analogs of these eigenvector behaviors in random matrix theory. Estimates from Green's functions \cite{isotropic,ESY_local,ErdYauYin2012Rig, KY_isotropic} showed delocalization of eigenvectors; namely, that the maximum entry of the eigenvector is of order close to $\frac{1}{N}$. These results have been strengthened in \cite{BouYau2017} to show Gaussian fluctuation for individual eigenvector entries; this is the appropriate analog of QUE for random matrices. The paper \cite {BouYau2017} shows
\begin{equation}
    \sqrt{N} \langle u_i, q \rangle  \to \mathcal{N}(0,1),
\end{equation}
i.e., the inner product of an eigenvector with a fixed vector approaches a standard normal random variable. Other results regarding proving QUE results include \cite{LP2021,Benigni2020,BouYau2017,PartI,XuYangYauYin}.  The paper \cite{Marcinek_thesis} studied the correlation of a small number of eigenvector entries ($\mathcal O(N^{\epsilon})$), and showed joint Gaussian behavior on these small windows. These works used was the eigenvector moment flow equations derived from Dyson Brownian motion. However, these equations are very difficult to analyze and do not yet give a complete description of the eigenvector statistics.  

}
The random matrix analog of eigenstate thermalization was studied in \cite{cipolloni-erdos-schroder-2021} by G. Cipolloni, L. Erdős, and D. Schr{\"o}der. In this paper, the authors tried to establish more global results on the distribution of the eigenvector.
Namely, they were able to show, for a Wigner matrix, that, with overwhelming probability,
\begin{equation}
    \max_{i,j} | \langle u_i, A u_j \rangle - \delta_{ij} \langle A \rangle \rangle| \lesssim \frac{N^{\epsilon}}{\sqrt{N}},
\end{equation}
where the error in the right hand side is optimal. In what follows, we use the notation $ \langle \cdot, \cdot \rangle$ to denote inner product in vector computations or the normalized trace $\langle A \rangle := \frac{1}{N} \text{Tr}[A]$ as appropriate in context.  Some of these results were extended to prove the normality of the terms $\langle u_i, A u_j \rangle$ in \cite{cipolloni-erdos-schroder-2022-normality, cipolloni-erdos-schroder-2020-functionalCLT} and multi-resolvent local laws in  \cite{cipolloni2022optimal}.

Rather than using the eigenvector moment flow, they directly studied more global quantities like $\Lambda:=\Ex{ GAGB}$, where $G = (H-z)^{-1}$ is the Green's function of the Wigner random matrix $H$, while $A$ and $B$ are arbitrary matrices. These quantities reveal more information about the correlation of eigenvectors on larger scales and, furthermore, are easier to manipulate analytically. The method of this work involved using the cumulant expansion to form a self-consistent equation for $\Lambda$. The details of the cumulant expansion procedure meant that the results of \cite{cipolloni-erdos-schroder-2021} were restricted to random matrices of Wigner type.

{In the paper \cite{cipolloni2022optimal}, multi-resolvent local laws for Wigner matrices were considered. They derived a hierarchy of equations to get more detailed estimates for traces of high powers of the form $(GA)^k$ which can also accommodate different traceless observables and handle them uniformly in all choices of observables. These results were expanded in the recent works \cite{cipolloni2022rank, cipolloni2023optimal,cipolloni2023gaussian,benigni2022fluctuations}: In \cite{cipolloni2022rank}, general local law for Wigner matrices which optimally handles observables of arbitrary rank were shown; thus, the paper unifies the averaged and isotropic local laws.
\cite{cipolloni2023gaussian} establishes the Eigenstate Thermalisation Hypothesis and Gaussian fluctuations for Wigner matrices with an arbitrary deformation. In \cite{cipolloni2023optimal}, the authors prove an optimal lower bound on the diagonal overlaps of the corresponding non-Hermitian eigenvectors.
\cite{benigni2022fluctuations} derives Gaussian fluctuations and gives a analog of the Berry conjecture for random matrices. 
}

\subsection{Difficulties in the case of Generalized Wigner matrices}

A Wigner matrix is a generalization of the the GUE or appropriate Gaussian ensemble. All of the entries are independent and identically distributed (i.i.d.). Due to this nice symmetric structure, one might believe on an intuitive level, that all of the relevant eigenvalue and eigenvector statistics would match that of the corresponding Gaussian ensemble. Namely, the eigenvalues are distributed according to the sine kernel and, more relevant to our case, the eigenvectors are Haar distributed. 

In the context of Eigenstate Thermalization,  one can prove the following claim,
\begin{equation}
\Ex{GA_1GA_2} \approx m^2 \Ex{A_1A_2}.
\end{equation}
Here, $m$ is the solution to the semicircle equation $$ m^2 - zm +1=0,$$ and one has the approximation $G_{ii} \approx m$, for all diagonal entries $G_{ii}$ of the resolvent matrix. Thus, to leading order, one can derive the approximation in Eigenstate Thermalization by replacing the resolvent matrices $G$ by the approximation $mI$. In this way, there is seemingly little contribution from the off-diagonal entries of $G$. As such, this statistic is further evidence for the approximate Haar distribution of the eigenvector entries in a Wigner matrix.

A generalized Wigner matrix is an ensemble of random matrices where every entry has an independent entry, up to symmetry conditions, but each entry has a different value of the variance; thus, the entries are not i.i.d. If $W$ is our generalized Wigner matrix, then $\mathbb{E}[|W_{ij}|^2] = S_{ij}$, for some number $S_{ij}$. The only constraint that we have is the following normalization constraint $$ \sum_{j} S_{ij} =1, \forall i.$$ Even with this constraint in place, one still has the following leading order behavior of the entries of the resolvent,
$$
G_{ii} \approx m, |G_{ij}| \approx \frac{1}{N \text{Im}[z]} = o(1).
$$

However, in the context of generalized Wigner matrices, we obtain an entirely different result. We have instead,
\begin{equation} \label{eq:traceGenWig}
    \Ex{G(z_1)A_1G(z_2)A_2} \approx m(z_1) m(z_2) \Ex{A_1A_2} + m(z_1) m(z_2) \frac{1}{N}\sum_{\alpha, \beta} (A_1)_{\alpha \alpha} \left[S(I-m(z_1)m(z_2) \mathcal{C})^{-1}\right]_{\alpha \beta} (A_2)_{\beta\beta},
\end{equation}
where $\mathcal{C}_{\mu\nu} = S_{\mu\nu} - \frac{1}{N}$ for any $\mu, \nu \in [N]$.

To get the above expression, it is no longer possible to replace $G(z)$ by the most obvious approximation $G(z) \approx m(z)I$, even though the leading behavior entry of each of the entries in the resolvent $G$ is the same as that of the Wigner random matrices. The implication of this fact is that there are detailed correlations present in the distribution of the eigenvectors of the generalized Wigner ensemble that are not present in the Wigner ensemble. In particular, the distribution of the eigenvectors of the generalized Wigner ensemble are far from Haar distrubuted.
Furthermore, the covariances of the terms $\langle u_i ,A u_j \rangle$ would depend on the eigenvector indices $i$ and $j$, while for the pure Wigner matrix, the covariance structure would be homogeneous in $i$ and $j$. We also remark here that this is only an effect you see in full rank matrices $A$; in the context of QUE with finite rank matrices (or even $N^{\epsilon}$ rank matrices for $\epsilon <1$), there is no difference in the covariance structure of eigenvectors for pure Wigner matrices and generalized Wigner matrices.

When coming to the proof of equation \eqref{eq:traceGenWig}, the main difficult is a presence of a more complicated term during the derivation of the self-consistent equation for the quantity $\Lambda$. Namely, if we consider the case of computing $\Ex{G A_1 G A_2}$ and $A_1,A_2$ are both traceless matrices, we have to deal with a term of the following form,
\begin{equation}
     \frac{1}{N}  \sum_{i, j} S_{ij} (G(z_1) A_1 G(z_2))_{jj} (G(z_2) A_2)_{ii}.
\end{equation}

In the Wigner case, we have that $S_{ij}=\frac{1}{N}$ for all $i$ and $j$. Thus, the above quantity can be simplified as,
\begin{equation} \label{eq:Sterm}
     \frac{1}{N} \sum_{i, j} S_{ij} (G(z_1) A_1 G(z_2))_{jj} (G(z_2) A_2)_{ii}= \Ex{G(z_1) A_1 G(z_2)} \Ex{G(z_2) A_2}.
\end{equation}

Now, since we have the heuristic that $\Ex{G(z_2) A_2} \approx m(z_2) \Ex{A_2} =0$ and $m$ is the Stieltjes transform for the  semicircle distribution, we can believe that the term above is merely a lower order term that should not complicate the analysis.

However, when $S_{ij} \ne \frac 1 N$ uniformly, there is no longer any way to write it as a product of traces. As such, it seems like using the fact that $A_1$ and $A_2$ are both traceless do not seem to give any cancellations. Indeed, if we take the approximation $G(z_i) \approx m(z_i)I$, we might guess that the term in \eqref{eq:Sterm} is at least as large as,
\begin{equation}
    \frac{1}{N} m^2 \sum_{i,j} S_{ij} (A_1)_{jj} (A_2)_{ii}.
\end{equation}
We cannot hope for the quantity above to be of smaller order.

The fact that the contribution of the term \eqref{eq:Sterm} presents us with two problems. The first issue is to actually determine the value to leading order. The second is to actually present this term in such a way that we get a closed equation. As we have mentioned earlier, in the Wigner case, these terms can be presented as products of traces; this means that we can derive closed equations just involving these products of traces. Without a closed equation, we cannot hope to analyze the resulting self-consistent equation; thus, it is of paramount importance to rewrite this term in a manner that is amenable to analysis. 
Our first main step is to write such terms as a product of traces by carefully decomposing the covariance matrix $S_{ij}$. By taking the square root, we have that,
\begin{equation*}
    S_{ij} = \sum_{\mu} \tilde{S}_{i \mu} \tilde{S}_{\mu j}.
\end{equation*}
With this decomposition in hand, we can rewrite,
\begin{equation} 
\begin{aligned}
    \frac{1}{N} \sum_{i, j} S_{ij} (G (z_1)A_1 G(z_2))_{jj} (G(z_2) A_2)_{ii} & = \frac{1}{N} \sum_{i,j, \mu} \tilde{S}_{i \mu} \tilde{S}_{j \mu} (G(z_1) A_1 G(z_2))_{jj} (G(z_2) A_2)_{ii} \\ &= \frac{1}{N} \sum_{\mu} \Ex{G(z_2) A_2 N \dg\tilde{S}_{\mu}} \Ex{
    G(z_1) A_1 G(z_2) N \dg \tilde{S}_{\mu}}.
\end{aligned}
\end{equation}

Here, $\dg \tilde{S}_{\mu}$ is the diagonal matrix whose $i$th entry is given by $\tilde{S}_{i \mu}$. The above expression looks like a more closed expression, due to the fact that we have written the above as a product of traces; however, we still need to consider traceless matrices if we actually want to consider eigenstate thermalization. 

An immediate solution here is to consider the traceless parts of the matrices $A_2 N \dg \tilde{S}_{\mu}$ and $N \dg \tilde{S}_{\mu}$, but this is still not closed since we keep introducing new traceless matrices of the form $A_2 N \dg \tilde{S}_{\mu}$. {The result of this procedure is to generate a chain of equations relating the $\Lambda_{A}$ of certain matrices $A$ to $\Lambda_{B}$ of other matrices $B$. At each step of this procedure, the hierarchy of matrices considered grows rapidly, and it is not clear that this chain would lead to an effective bound.  For instance, the matrices at level $k+1$ consists of any product of two matrices at level $k$. If one did not have precise control of appropriate prefactors when deriving the inequalities, then it would be impossible to derive useful information. For example, if one were to try to prove the case for non-diagonal matrices at the very beginning, one would have to deal with a cubic term that cannot be controlled via iteration. We circumvent this issue by first proving estimates for diagonal matrices, in which one can apply improved local law estimates, in order to have optimal estimates for the diagonal $\Lambda_S$. These estimates are key inputs for deriving bounds on $\Lambda$ for the general case of non-diagonal matrices.}  The main achievement of Sections \ref{sec:diagM} and \ref{sec:generalM} of this manuscript is to derive this system of inequalities. 

{
{\it Acknowledgement:} The authors are grateful to L{\'a}szl{\'o} Erdős for the useful comments and to Horng-Tzer Yau for the valuable discussions.}

\subsection{Conventions and Notation}

We use the notation $\prec$ to indicate stochastic domination (see also e.g. \cite{cipolloni-erdos-schroder-2021}) indicating a bound with very high probability up to a factor $N^\epsilon$ for any small $\epsilon > 0$. If
\begin{equation}
    X = \left( X^{(N)}(u) \mid N \in \mathbb N, u \in U^{(N)}\right) \quad and \quad    Y = \left(Y^{(N)}(u) \mid N \in \mathbb N, u \in U^{(N)}\right) 
\end{equation}
are families of non-negative random variables indexed by $N$, and possibly some parameter $u$, then we say
that X is stochastically dominated by $Y$, if for all $\epsilon, D>0$ we have
$$\sup_{u \in U^{(N)}} \mathbb P\left(X^{(N)}(u) \geq N^\epsilon Y^{(N)}(u) \right) \leq N^{-D}$$
for large enough $N \geq N_0(\epsilon, D)$. In this case we use the notation $X \prec Y$ or $X = \mathcal O_\prec (Y)$.

For any $N\times N$ matrix $M$ we use the following notation for the normalized trace:
\begin{equation}
    \Ex{M} = \frac{1}{N} \tr M.
\end{equation}



\section{Main Results}

In this paper, we consider generalized Wigner matrices. Namely, these are Hermitian matrices, where each entry is independent, but they are allowed to have different variances. Our normalization condition on the variances is that $ \sum_{j} S_{ij} =1, \forall i$, where $S_{ij}$ is the variance of the $(i,j)$th entry. We let $S=[S_{ij}]$ denote the full covariance matrix of our generalized Wigner matrix. For a more formal definition, see Section 2 of \cite{ErdYauYin2012Rig}. To simplify our analysis, we need the following assumption on the entries of the square root of $S$.  It is easy to see that the following assumption can hold for small perturbations of the covariance matrix of a Wigner matrix. 







\begin{assumption}\label{asmp:squareroot}
Let $\tilde{S}$ be the square root of $S$. We assume that there is a constant $C>0$ such that for all $i,j$, we have,
    \begin{equation}
    \frac{1}{C} \frac{1}{N} \le \tilde{S}_{ij}  \le C \frac{1}{N}.  
    \end{equation}

    One can check that this condition holds if $S$ were the matrix whose entries were all $\frac{1}{N}$. 

\end{assumption}

From the paper \cite{ErdYauYin2012Rig}, we have the following a-priori estimate on the behavior of the Green's function of our generalized Wigner matrices. These will be used multiple times in the proof.

\begin{theorem}[{\cite[Theorems 2.1, 2.2]{ErdYauYin2012Rig}}]
Let $W$ be a generalized Wigner matrix. We assume that the probability distributions of each entry of $W$ have a uniform sub-exponential decay. Then the following estimates hold:
\label{thm:lowcallaw}
\begin{equation}
\label{eq:stronglocallaw}
    \| G(z) - m I \|_{\max} \prec \sqrt{\frac{\Im m(z)}{N\eta}}  + \frac{1}{N\eta} \quad \text{ for all }|E| \leq 5, \eta \geq N^{-1 + \epsilon}\,.
\end{equation}
\begin{equation}
\label{eq:rig}
    |\lambda_i - \gamma_i| \prec \min(i,N-i+1)^{-1/3}N^{-2/3}\text{ for all }1\leq i\leq N\,.
\end{equation}
\end{theorem}

With these preliminaries in hand, we can state our main result.

\begin{theorem} \label{thm:EigenstateThermHyp}
Let $M$ be a Hermitian matrix with trace $0$ and bounded norm $||M||\le 1$. Let $W$ be our random generalized Wigner matrix as we have previously constructed; let $\lambda_1 \ge \lambda_2 \ldots \ge \lambda_N$ be its eigenvalues with corresponding eigenvectors $u_1,u_2,\ldots,u_n$. With overwhelming probability for any $\xi>0$, we can derive the following estimates:
\begin{equation}
    \max_{i,j}|\langle u_i, M u_j \rangle| + \max_{i,j} |\langle u_i, M \overline{u}_j \rangle| \le \frac{N^{\xi}}{\sqrt{N}}\,.
\end{equation}
\end{theorem}

We study the entrywise maximum through the following intermediate quantity $\Xi_M$, as in \cite{cipolloni-erdos-schroder-2021}. $\Xi_M$ computes averaged versions of the quantity in interest in Theorem \ref{thm:EigenstateThermHyp}. 

\begin{definition}
Let $A$ be a matrix and $J \in \N$. We define $\Xi_A,\overline{\Xi}_A$ as,
\begin{equation}
\begin{aligned}
    &\Xi_{A}(J):= \frac{N}{(2J)^2} \max_{i_0,j_0} \sum_{|i-i_0| \le J} \sum_{|j-j_0| \le J} |\langle u_i, A u_j \rangle|^2\,,\\
    & \overline{\Xi}_{A}(J):= \frac{N}{(2J)^2} \max_{i_0,j_0} \sum_{|i -i_0|\le J} \sum_{|j-j_0| \le J} |\langle u_i,A \overline{u}_j \rangle|^2\,.
\end{aligned}
\end{equation}

We will omit the dependence of $\Xi_A$ on $J$ when the context is clear.
\end{definition}

In contrast to the paper \cite{cipolloni-erdos-schroder-2021}, in which the authors could derive a self-consistent equation consisting of only one matrix, we have the relate the quantities $\Xi_M$ of different families of matrices to each other. We now introduce the following classes of deterministic matrices of interest.
\begin{align}
  & \mathbb{M}_0 := \{ N\dg \tilde{S}_\mu \}_{1 \le \mu \le N}\,, \quad  \mathbb{M}_1 := \{I, M\} \cup  \mathbb{M}_0\,,\\
&    \mathbb{M}_k := \left\{ B_1B_2 : B_1,B_2 \in \mathbb{M}_{k-1} \cup \mathbb{M}_{k-1}^\circ , 1\le \mu\le N \right\} ~~\text{for $k \ge 2$,}\\
&\mathbb{M}_k^\circ := \{B - \Ex{B} : B \in \mathbb{M}_k \},\\
   & \Lambda_k := \max_{B \in \mathbb M_k^\circ}\Xi_B+ \max_{B \in \mathbb M_k}\overline{\Xi}_B + 1\,.
\end{align}

The bound in the following lemma is a simple consequence of our definitions.
\begin{lemma}
$\sup_{B \in {\mathbb M_k \cup \mathbb M_k^\circ}}\|B\|_{l^2 \to l^2} \leq (\sup_{B \in {\mathbb M_1 \cup \mathbb M_1^\circ}}\|B\|_{l^2 \to l^2})^{2^k}$.
\end{lemma}


Our main result Theorem \ref{thm:EigenstateThermHyp} is an easy corollary of the following result on the size of the control parameters $\Xi_M$ and $\overline{\Xi}_M$.
\begin{theorem} \label{thm:Xicontrolparam}
Fix $1 > \epsilon > 0$ and $J \ge N^{\epsilon}. $Let $W$ be our generalized Wigner matrix as before and let $M$ be a trace-less Hermitian matrix with bounded norm $||M||\le 1$. Then, we have the following estimates
\begin{equation}
    \Xi_M(J), \overline{\Xi}_M(J) \prec 1.
\end{equation}

\end{theorem}

\begin{proof}[Proof of Theorem \ref{thm:EigenstateThermHyp}]

From Theorem \ref{thm:Xicontrolparam}, we know that $\Xi_{A} \prec 1$. From this fact then necessarily, for any $i,j$, we must have that $|\langle u_i, A u_j \rangle|^2 \prec \frac{(2J)^2}{N} $. By taking square roots of both sides, we are done.
\end{proof}

Furthermore, the function $\Xi_A$ can be related to more standard functions of the resolvent of our Wigner matrix; $ G(z):=(W-z)^{-1}$. In what appears later, if we are considering a matrix product, then we let $\langle \cdot \rangle$ denote the normalized trace of the matrix under consideration.

For example, consider the following expression with $z_i= E_i + \ti \eta_i$:
\begin{equation} \label{eq:prodImG}
    \Ex{\Im G(z_1) A \Im G(z_2) A^*} = \frac{1}{N} \sum_{i,j} \frac{|\langle u_i, A u_j \rangle|^2 \eta_1\eta_2}{((\lambda_i - E_1)^2 + \eta_1^2)((\lambda_j -E_2)^2 + \eta_2^2)} .
\end{equation}

The following lemma explicitly writes out the relations between $\Xi_{A}$ and the quantity presented in equation \eqref{eq:prodImG}. 
\begin{lemma}
\label{GAGA&Xi}
{ Fix $E_1 = \gamma_{i_0}$, $E_2 = \gamma_{j_0}$ and $J \ge N^{\epsilon}$, where the $\gamma_{i}'s$ represent the classical eigenvalue locations of the $i$th eigenvalue. Choose $\eta_1$ and $\eta_2$ so that the following equation holds $J= N \eta_i \rho_i$, where $\rho_i = \Im m(E_i + \ti \eta_i)$. Then, we have the following claim,}
\begin{equation}
\begin{aligned}
    &\frac{N}{(2J)^2} \sum_{\substack{|i-i_0| \le J \\ |j-j_0| \le J}} \left|\langle u_i,A u_j \rangle\right|^2 \prec \frac{\Ex{\Im G(z_1) A \Im G(z_2) A^*}}{\rho_1 \rho_2}  \prec\frac{N}{(2J)^2} \sum_{\substack{|i-i_0| \le J \\ |j-j_0| \le J}} \left|\langle u_i,A u_j \rangle\right|^2,\\
    &\frac{N}{(2J)^2} \sum_{\substack{|i-i_0| \le J \\ |j-j_0| \le J}} \left|\langle u_i,A \overline{u}_j \rangle\right|^2 \prec \frac{\Ex{\Im G(z_1) A \Im G^t(z_2) A^*}}{\rho_1 \rho_2}  \prec \frac{N}{(2J)^2}\sum_{\substack{|i-i_0| \le J \\ |j-j_0| \le J}} \left|\langle u_i,A \overline{u}_j \rangle\right|^2.
\end{aligned}
\end{equation}
\end{lemma}

\begin{proof}
This is a consequence of eigenvalue rigidity \eqref{eq:rig} for generalized Wigner matrices. See \cite[Lemma 3.2]{cipolloni-erdos-schroder-2021}.
\end{proof}






Our basic tool for deriving a self consistent equation for quantities of the form appearing in equation \eqref{eq:prodImG} is integration by parts. One of our main error terms produced by this integration by parts procedure is the following renormalized term. 

\begin{definition}[Renormalized Matrix Products] \label{defn:Renorm}
Given a matrix product of the from $f(W) W g(W)$, we can define the renormalized matrix product $\underline{f(W) W g(W)}$ as,
\begin{equation}\label{eq:renorm}
    \underline{f(W) W g(W)} := f(W) W g(W) - \mathbb{E}_{\tilde{W}} (\partial_{\tilde{W}}f)(W) \tilde{W} g(W) - \mathbb{E}_{\tilde{W}} f(W) \tilde{W} (\partial_{\tilde{W}} g)(W).
\end{equation}

The derivative $\partial_{\tilde{W}}f= \sum_{i,j} \tilde{W}_{ij} \partial_{ij}f$, where $\partial_{ij}f$ is the standard partial derivative of $f$ with respect to the $ij$th matrix entry and $\tilde W$ is an independent copy of $W$.
\end{definition}

\begin{remark}
    The terms subtracted in \eqref{eq:renorm} are the first order terms in the integration by parts of $f(W)Wg(W)$ with respect to the middle $W$ in the product.
\end{remark}

Our final main lemma computes the size of the renormalized term for our relevant quantities of interest.

\begin{lemma}\label{lem:underline}
Let $W$ be a generalized Wigner matrix satisfying the conditions lined out in Assumption \ref{asmp:squareroot}. Suppose for $i \in \{1,2\}$ $z_i  \in \mathbb{C}\setminus \mathbb{R}$, $\eta_i = \left|\Im z_i\right|$, $\rho_i = \Im m_i$, $L = \min{\left|N\eta_i\rho_i\right|}$, $\eta_* = \min(\eta_1,\eta_2)$. Then, we have the following estimates.



For $G_i \in \{G(z_i), G^*(z_i), G^t(z_i), \Im G(z_i) \}$ and $A \in \mathbb{M}_k^\circ$
\begin{equation}
    \left|\Ex{\underline{WG_iA}}\right| \prec \frac{\rho_i \Lambda_k}{\sqrt{NL}}\,.
\end{equation}
For $G_i \in \{G(z_i), G^*(z_i), G^t(z_i)\}$ and $A \in \mathbb{M}_k^\circ$
\begin{equation}
\begin{aligned}
    \left|\Ex{\underline{WG_1 G_2 A}}\right| &\prec \frac{\Lambda_k}{L\sqrt{\eta_*}}\,,\quad  \left|\Ex{\underline{W G_1 \Im G_2 A}}\right| \prec \frac{\rho_2\Lambda_k}{L\sqrt{\eta_*}}\,, \\
    \left|\Ex{\underline{W\Im G_1 G_2 A}}\right| &\prec \frac{\rho_1\Lambda_k}{L\sqrt{\eta_*}}\,,\quad \left|\Ex{\underline{W\Im G_1 \Im G_2 A}}\right| \prec \frac{\rho_1\rho_2\Lambda_k}{L\sqrt{\eta_*}}\,.
\end{aligned}
\end{equation}
For $G_i \in \{G(z_i), G^*(z_i), G^t(z_i)\}$, $A_1 \in \mathbb{M}_k^\circ$ and $A_2 \in \mathbb{M}_l^\circ$
\begin{equation}
\begin{aligned}
    |\Ex{\underline{W G_1 A_1 G_2 A_2}}| &\prec \frac{\Lambda_k\Lambda_l}{\sqrt{L}}\,,\quad
    |\Ex{\underline{W G_1 A_1 \Im G_2 A_2}}| \prec \frac{\rho_2\Lambda_k\Lambda_l}{\sqrt{L}}\,,\\
        |\Ex{\underline{W \Im G_1 A_1 G_2 A_2}}| &\prec \frac{\rho_1\Lambda_k\Lambda_l}{\sqrt{L}}\,,\quad
    |\Ex{\underline{W \Im G_1 A_1 \Im G_2 A_2}}| \prec \frac{\rho_1\rho_2\Lambda_k\Lambda_l}{\sqrt{L}}\,.
\end{aligned}
\end{equation}
\end{lemma}

\section{Proof of Theorem \ref{thm:Xicontrolparam} for $M$ diagonal}
\label{sec:diagM}

To prove Theorem \ref{thm:Xicontrolparam}, we need the bounds on  expressions of the form $\Ex{\Im G A \Im G A^*}$ and $\Ex{\Im G A \Im G^t A^*}$ in terms of $\Lambda_k$. To do this, we first have to study simpler expressions like $\Ex{GA}$, $\Ex{GGA}$, $\Ex{GAGA}$, etc.

Throughout Sections \ref{sec:diagM} and \ref{sec:generalM} we use the following notation. Let $z_i \in \mathbb{C} \setminus \R$, $G_i\in \{G_i(z), G^*_i(z)\}$, $\eta_i = \left|\Im z_i\right|$, $\rho_i = \Im m_i$, $L = \min{\left|N\eta_i\rho_i\right|}$, $\eta_* = \min \eta_i$. In the case that the studied expression has a single resolvent, we omit the index $i$.  

In this section we assume that $M$ is diagonal and, thus, all matrices in the families $\mathbb{M}_k$ and $\mathbb{M}_k^\circ$ are diagonal.

\subsection{Bounds on $\Ex{GA}$}
\begin{lemma} \label{lem:bndGA}
Let $A \in \mathbb M_k^\circ$. Then, we have that,
\begin{equation}
    \left|\Ex{GA}\right| \prec \frac{\sqrt{\rho}\Lambda_k}{N\sqrt{\eta}} = \frac{\rho\Lambda_k}{\sqrt{NL}},
\end{equation}
and, therefore
\begin{equation}
    \left|\Ex{\Im GA}\right| \prec \frac{\sqrt{\rho}\Lambda_k}{N\sqrt{\eta}} = \frac{\rho\Lambda_k}{\sqrt{NL}}.
\end{equation}

\end{lemma}
\begin{proof}

First, we start with the following identity,
\begin{equation}\label{eq:G-to-WG-identity}
    G = mI - m WG - m^2G.
\end{equation}
Multiplying this by the matrix $A$, we get,
\begin{equation}
    GA = mA - m WGA - m^2GA.  
\end{equation}
We replace the term $WGA$ by the renormalization from Definition \ref{defn:Renorm} and derive,
\begin{equation}
      (GA)_{ik} = mA_{ik} - m (\underline{WGA})_{ik} + m  \sum_j S_{ij}(G_{jj} - m )(GA)_{ik}.
\end{equation}
Taking the trace of this expression, we have that, for traceless matrices $A$, that
\begin{equation} \label{eq:fundGA}
    \langle GA \rangle = - m \langle\underline{WGA} \rangle + m \frac{1}{N} \sum_{i,j} S_{ij}(G_{jj} -m) (GA)_{ii}.
\end{equation}
We introduce the splitting $S_{ij} = \sum_{\mu} \tilde{S}_{i \mu} \tilde{S}_{\mu j}$ on the last term and we further introduce the traceless part $\tilde{S}_{i \mu} = \tilde{S}^{\circ}_{i \mu} + \frac{1}{N}$.
\begin{equation} \label{eq:longGA}
\begin{aligned}
    \frac{1}{N} \sum_{i,j} S_{ij}(G_{jj} - m) (GA)_{ii} &= \frac{1}{N} \sum_{\mu} \sum_{i,j} \tilde{S}_{i \mu}  (GA)_{ii} \tilde{S}_{\mu j} (G_{jj}-m)\\
    &= \frac{1}{N} \sum_{\mu}\sum_{i,j} \tilde{S}_{i \mu} (GA)_{ii} \tilde{S}^{\circ}_{\mu j} (G_{jj} - m) + \frac{1}{N^2} \sum_{\mu} \sum_{i,j}\tilde{S}_{\mu i} (GA)_{ii}  (G_{jj} - m)\\
    &= \frac{1}{N} \sum_{\mu} \Ex{G A N \dg\tilde{S}_{ \mu}} \Ex{G N\dg\tilde{S}_{\mu}^\circ} + \Ex{GA} \Ex{G - mI}\\
    &= \frac{1}{N} m \sum_{\mu} \Ex{A N \dg\tilde{S}_{\mu}} \Ex{G N\dg\tilde{S}_{\mu}^\circ} \\
    &+  \frac{1}{N} \sum_{\mu} \Ex{(G - mI) A N \dg\tilde{S}_{\mu}} \Ex{G N\dg\tilde{S}_{\mu}^\circ} + \Ex{GA} \Ex{G- mI}.
\end{aligned}
\end{equation}
We will specialize $A = N\dg\tilde{S}^{\circ}_{\nu}$ to get a certain system of equations. First, we have
\begin{equation}
\begin{aligned}
    \frac{1}{N} \sum_{i,j} S_{ij}(G_{jj} - m) (G N\dg\tilde{S}^{\circ}_{\nu})_{ii}= \sum_{\mu} C_{\nu\mu} \langle G N \dg\tilde S^{\circ}_{\mu}\rangle,
\end{aligned}
\end{equation}
where the coefficients $C_{\nu\mu}$ are
\begin{equation}
\begin{aligned}
    C_{\nu\mu} &= \frac{m}{N} \Ex{N \dg\tilde{S}^{\circ}_{\nu}N \dg\tilde{S}_{\mu}} + \frac{1}{N} \Ex{(G- mI) N\dg\tilde{S}^{\circ}_{\nu} N \dg\tilde{S}_{\mu}} + \Ex{G- mI}\delta_{\nu\mu}  \\&= m \left[S_{\nu\mu} - \frac{1}{N} \right]  + \mathcal O_\prec\left(\frac{1}{N^{3/2} \eta^{1/2}}\right) + \delta_{\nu\mu}\mathcal O_\prec\left(\frac{1}{N\eta}\right).
\end{aligned}
\end{equation}
In the line above, we used the local law to bound the diagonal entries of $G- mI$ by $\frac{1}{\sqrt{N \eta}}$.  Furthermore, $N \dg\tilde{S}_{\mu}^\circ N \dg\tilde{S}_{\nu}$ is a diagonal matrix with $\mathcal O(1)$ entries. Thus, we see that,
\begin{equation}
\frac{1}{N} \langle (G- mI) N \dg\tilde{S}^{\circ}_{\nu} N\dg\tilde{S}_{\mu} \rangle=\frac{1}{N^2} \sum_{i=1}^N (G - mI)_{ii} [N\dg\tilde{S}^{\circ}_{\nu} N\dg\tilde{S}_{\mu}]_{ii} = \mathcal O_\prec\left(\frac{1}{N^{3/2} \eta^{1/2}} \right).
\end{equation}

Placing all of these estimates back into the equation \eqref{eq:fundGA} for specialized values of $A = N\dg\tilde{S}_{\nu}^\circ$, we have,
\begin{equation} \label{eq:somematrix}
    (I - C) \Ex{G N\dg\tilde{S}^{\circ}} = - m \Ex{\underline{W G N\dg\tilde{S}^{\circ}}}.
\end{equation}
Here, $\langle G \dg \tilde{S}^{\circ} \rangle$  and $\langle \underline{W G \dg\tilde{S}^{\circ} \rangle}$ is a shorthand for the column vector constructed using these terms for the matrix $\dg \tilde{S}^{\circ}_{\mu}$ for each $\mu$. 

This finally gives us,
\begin{equation}
    \Ex{G N\dg\tilde{S}^{\circ}} = -m (I - C)^{-1} \Ex{\underline{WG N\dg\tilde{S}^{\circ}}}.
\end{equation}

\begin{lemma}\label{lem:matrix-inversion}
Assume that we have $S_{ij} \ge \frac{c}{N}$ for all values i,j.
The largest eigenvalue of $S - \frac{1}{N}1^T 1$ in absolute value is bounded from above by $1 -c$ .
\end{lemma}
\begin{proof}
The matrix $S - \frac{1}{N} 1^T 1$ can be decomposed as follows,
\begin{equation}
    S - \frac{1}{N}1^T 1 = S_2 + \frac{c -1}{N} 1^T 1 ,
\end{equation}
where $1$ is the row vector with all entries equal to $1$.
All of the entries of $S_2$ are positive; furthermore, the sum over each row and each column is bounded by $1- c$. This shows that the $l_2 \to \l_2$ operator norm of the matrix $S_2$  is less than $1-c$. If we look at the orthogonal space to the vector $1$, we see that $\sup_{\langle v, 1 \rangle =0} |v \left[ S - \frac{1}{N} 1^T 1\right]  v^T| \le 1-c$.

Furthermore,the vector $1$ is an eigenvector of the matrix $S - \frac{1}{N} 1^T 1$ with eigenvalue $c-1$. Thus, the largest eigenvalue of $S $ is bounded by,
$$
\max\left(|c-1|,  \sup_{\langle v, 1 \rangle =0} v \left[ S - \frac{1}{N}1^T 1 \right] v^T\right) \le 1-c,
$$
as desired. 
\end{proof}

Because of the above lemma, along with the fact that $|m|<1$ (as the Stieltjes transform of the semicircle distribution), we know that the largest eigenvalue of $C$ is bounded from above in absolute value by $1-c$. Thus, the inverse $(I - C)^{-1} \langle \underline{WNG \dg\tilde{S}^{\circ}} \rangle$ is well-defined and bounded in $l_2$ vector norm by $\sqrt{N} \frac{\sqrt{\rho}\Lambda_1}{\sqrt{NL}}.$

Placing this estimate back into the equation for an individual row in \eqref{eq:somematrix}, we find that,
\begin{equation}
    \Ex{G  N \dg\tilde{S}_{\mu}^\circ} = \left\langle C_{\mu}, \Ex{G N \dg\tilde{S}^{\circ}} \right\rangle - m \Ex{\underline{WG N\dg\tilde{S}_{\mu}}}.
\end{equation}

$C_{\mu}$ is the $\mu$th row of the matrix $C$. 
Now, $C_{\mu}$ is bounded in $l_2$ norm by $\mathcal{O}\left(\frac{1}{\sqrt{N}}\right)$. By the Cauchy-Schwartz inequality $ \langle C_{\mu}, \langle GN \dg\tilde{S}^{\circ} \rangle \rangle $ can be bounded by $\prec \frac{\sqrt{\rho}\Lambda_1}{\sqrt{NL}}$.

This shows that the entries,
\begin{equation} \label{eq:GdiagS}
    |\langle G N\dg\tilde{S}_{\mu} \rangle| \prec \frac{\sqrt{\rho}\Lambda_1}{\sqrt{NL}}.
\end{equation}

At this point, we can return to an analysis for general matrices $A$. From the equation \eqref{eq:longGA}, we see that,
\begin{equation}
\begin{aligned}
    \langle GA \rangle [ 1- \langle G - mI \rangle] &= - m\langle{\underline{WGA}}\rangle + \frac{m^2}{N} \sum_{\mu} \langle A (N \dg(\tilde{S}^{\mu})) \rangle \langle G N \dg(\tilde{S}_{\mu}^{\circ})\rangle\\
    &+ \frac{1}{N} \sum_{\mu} \langle (G- mI) A (N \dg\tilde{S}_{\mu})\rangle \langle G N \dg(\tilde{S}_{\mu}^{\circ}) \rangle\\
    &\prec   \frac{ \sqrt{\rho} \Lambda_k }{\sqrt{NL}}  
\end{aligned}
\end{equation}
Our earlier estimates on $\langle G N \dg\tilde{S}^{\circ}_{\mu} \rangle$ as well as on $\langle \underline{WGA} \rangle$ ensure that the right hand side is $\prec \frac{\Lambda_k}{N \sqrt{\eta}}$. Here, we specifically used the fact that $A$ was a diagonal matrix, so that we can write,
$$
\Ex{(G- mI) A N \dg\tilde{S}_{\mu}} = \frac{1}{N} \sum_{i=1}^N (G - mI)_{ii} A_{ii} [N \dg\tilde{S}_{\mu}]_{ii}.
$$
All the terms $A_{ii}$ and $(N\dg\tilde{S}_{\mu})_{ii}$ are $\mathcal{O}(1)$.  Furthermore, $|G- mI|_{ii} = \mathcal{O}_{\prec} \left( \frac{1}{\sqrt{N\eta}}\right)$. Thus, the normalized trace considered above is $\mathcal{O}(1)$.

We can easily divide by $1 - \langle G - mI \rangle = 1 -  o(1)$ to derive the same error estimate for $\langle GA \rangle.$  
\end{proof}

\subsection{Bounds on $\Ex{G G A}$}

\begin{lemma}
\label{GGAdiag}
Let $A \in \mathbb M_k^\circ$. Then
\begin{equation}\label{eq:GGAdiag}
    |  \Ex{G_1G_2A} |\prec \frac{\Lambda_k }{L\sqrt{\eta_*}},
\end{equation}
\begin{equation}\label{eq:GImGAdiag}
    |\Ex{G_1 \Im G_2 A}| 
    \prec \frac{\rho_2\Lambda_k}{L\sqrt{\eta_*}},
\end{equation}
\begin{equation}\label{eq:ImGImGAdiag}
    |\Ex{\Im G_1 \Im G_2 A}| 
    \prec \frac{\rho_1\rho_2\Lambda_k}{L\sqrt{\eta_*}}.
\end{equation}
\end{lemma}

\begin{proof}

First, we use identity \eqref{eq:G-to-WG-identity} on $G_1$ and replace $WG_1$ by its renormalization from Definition \ref{defn:Renorm} as follows.
\begin{equation}
\begin{aligned}
    (G_1 G_2 A)_{ik} &= \sum_{l = 1}^N \left( m_1 \delta_{il} - m_1 (\underline{WG_1})_{il} + m_1 \sum_{j = 1}^N S_{ij}((G_1)_{jj} - m_1)(G_1)_{il}\right)(G_2A)_{lk} \\
    &= m_1(G_2A)_{ik} - m_1(\underline{WG_1}G_2A)_{ik} + m_1 \sum_{j = 1}^N S_{ij}((G_1)_{jj} - m_1)(G_1G_2A)_{ik}.
\end{aligned}
\end{equation}
From Definition \ref{defn:Renorm}, we can see that
\begin{equation}
    (\underline{WG_1G_2A})_{ik} = (\underline{WG_1}G_2A)_{ik} + \sum_{j=1}^N S_{ij} (G_1G_2)_{jj} (G_2A)_{ik}.
\end{equation}
Thus,
\begin{equation} \label{eq:GGA}
\begin{aligned}
    \Ex{G_1G_2A} &= m_1 \Ex{G_2A} - m_1 \Ex{\underline{WG_1G_2A}} + \frac{1}{N} m_1 \sum_{i,j=1}^N S_{ij} (G_1G_2)_{jj} (G_2A)_{ii} \\
    &+ \frac{1}{N} m_1 \sum_{i,j = 1}^N S_{ij}((G_1)_{jj} - m_1)(G_1G_2A)_{ii}.
\end{aligned}
\end{equation}

In the last two terms, we split $S$ as follows.
\begin{equation}\label{eq:G1G2A}
\begin{aligned}
    \Ex{G_1G_2A} &= m_1 \Ex{G_2A} - m_1 \Ex{\underline{WG_1G_2A}} \\
    &+ \frac{1}{N} m_1 \sum_{\mu = 1}^N \Ex{G_1G_2N\dg{\tilde{S}^\circ_\mu}} \Ex{G_2AN\dg{\tilde{S}_\mu}} + m_1 \Ex{G_1G_2} \Ex{G_2A} \\
    &+ \frac{1}{N} m_1 \sum_{\mu = 1}^N \Ex{G_1 N\dg{\tilde{S}_{\mu}^\circ}} \Ex{G_1G_2AN\dg{\tilde{S}_{\mu}}} + m_1 \Ex{G_1 - m_1} \Ex{G_1G_2A}.
\end{aligned}
\end{equation}

To bound the first term we use Lemma \ref{lem:bndGA} and get
\begin{equation}\label{eq:GGAbound-term1}
    \left|m_1\Ex{G_2A} \right| \prec \frac{\rho_2 \Lambda_k}{\sqrt{NL}}.
\end{equation}

Suppose $B = AN\dg{\tilde{S}_{\mu}}$ or $B = I$. We apply Cauchy-Schwarz to bound $\Ex{G_1G_2B}$.
\begin{equation}
    \left|\Ex{G_1G_2B}\right| \le \Ex{G_1 G_1^*}^{\frac{1}{2}} \Ex{G_2 B B^* G_2^*}^{\frac{1}{2}} \prec \frac{\sqrt{\rho_1 \rho_2}}{\sqrt{\eta_1 \eta_2}} \prec \frac{N\rho_1\rho_2}{L}.
\end{equation}
This gives us
\begin{equation}\label{eq:GGAbound-term2}
\begin{aligned}
    \left|\Ex{G_1 N\dg{\tilde{S}_{\mu}^\circ}} \Ex{G_1G_2AN\dg{\tilde{S}_{\mu}}}\right| \prec \frac{\Lambda_1}{N\sqrt{\eta_*}}\cdot \frac{N\rho_1\rho_2}{L} \prec \frac{\rho_1\rho_2\Lambda_1}{L\sqrt{\eta_*}}\,.
\end{aligned}
\end{equation}
and
\begin{equation}\label{eq:GGAbound-term3}
\begin{aligned}
    \left|\Ex{G_1G_2} \Ex{G_2A}\right| \prec \frac{\Lambda_k}{N\sqrt{\eta_*}}\cdot \frac{N\rho_1\rho_2}{L} \prec \frac{\rho_1\rho_2\Lambda_k}{L\sqrt{\eta_*}}
\end{aligned}
\end{equation}
Using this estimate, the estimate for $\Ex{GA}$ from above and $\Ex{G_1 - m_1} \prec \frac{1}{N\eta_1}$, we get
\begin{equation} \label{eq:G1G2Aint}
\begin{aligned}
    \left(1 + \mathcal{O}\left(\frac{1}{N\eta_1}\right)\right)\Ex{G_1G_2A} &= - m_1 \Ex{\underline{WG_1G_2A}} \\
    &+  \frac{1}{N} m_1 \sum_{\mu = 1}^N \Ex{G_1G_2N\dg{\tilde{S}^\circ_\mu}} \Ex{G_2AN\dg{\tilde{S}_\mu}}
    \\&+ \mathcal O_\prec\left(\frac{\Lambda_k \rho_1 \rho_2}{L \sqrt{\eta_*}}\right).
\end{aligned}
\end{equation}
We can bound $\langle \underline{W G_1 G_2 A} \rangle \prec \frac{\Lambda_k}{L\sqrt{\eta_*} }$ via Lemma \ref{lem:underline}.

Now we plug in $A = N\dg{\tilde{S}^\circ_\nu}$ to get the system of equations
\begin{equation}\label{eq:G1G2S-system}
    \left(I - C\right) \Ex{G_1G_2N\dg{\tilde{S}^\circ}} =  \mathcal{O}_\prec\left(\frac{\Lambda_1}{L\sqrt{\eta_*}}\right),
\end{equation}
where $C$ is a matrix with 
\begin{equation}
\begin{aligned}
    C_{\nu \mu} &= \frac{1}{N} m_1 \Ex{G_2 N\dg{\tilde{S}^\circ_\nu} N\dg{\tilde{S}_\mu}} + \mathcal{O}_\prec\left(\frac{1}{N\eta}\right)\delta_{\nu\mu} \\
    &= \frac{1}{N} m_1 \Ex{(G_2 - m_2) N\dg{\tilde{S}^\circ_\nu} N\dg{\tilde{S}_\mu}} + m_1 m_2 \left(S_{\nu\mu} - \frac{1}{N}\right) + \mathcal{O}_\prec\left(\frac{1}{N\eta}\right)\delta_{\nu\mu} \\
    &= m_1 m_2 \left(S_{\nu\mu} - \frac{1}{N}\right) + \mathcal{O}_\prec\left(\frac{1}{N\eta}\right)\delta_{\nu\mu} + \mathcal{O}_\prec\left(\frac{1}{N^{3/2}\eta^{1/2}}\right).
\end{aligned}
\end{equation}
To get the above estimates, we used the fact that $A= N\dg\tilde{S}^{\circ}_{\nu}$ has the better error bounds from \eqref{eq:GdiagS}.

Similarly to the proof of Lemma \ref{lem:bndGA} we use Lemma \ref{lem:matrix-inversion} to invert matrix $I-C$ in \eqref{eq:G1G2S-system} and get
\begin{equation}
    \left|\Ex{G_1G_2N\dg{\tilde{S}^\circ}}\right| \prec \frac{\Lambda_1}{L\sqrt{ \eta_*}}.
\end{equation}

Now, we can plug these estimates into equation \eqref{eq:G1G2Aint}.
In general, we see that we have,
\begin{equation}
\begin{aligned}
   &\left[1 - \mathcal{O}_\prec\left(\frac{1}{N \eta}\right)\right] |\langle G_1 G_2 A \rangle|  \prec  \Big|\frac{1}{N} \sum_{\mu=1}^N \langle G_1 G_2 N \dg\tilde{S}_{\mu} \rangle \langle G_2 A N \dg\tilde{S}_{\mu} \rangle  \Big|+ \frac{\Lambda_k}{L \sqrt{ \eta_*}} \\
   & \hspace{1 cm} \prec \Big|\frac{1}{N} \sum_{\mu =1}^N \langle G_1 G_2 N\dg\tilde{S}_{\mu} \rangle \left[\langle (G_2 - m_2) AN \dg\tilde{S}_{\mu} \rangle+ m_2 \langle AN \dg\tilde{S}_{\mu} \rangle\right] \Big|+ \frac{\Lambda_k}{L \sqrt{ \eta_*}} \\
   &\hspace{1 cm}\prec \frac{\Lambda_k}{L \sqrt{ \eta_*}}.
\end{aligned}
\end{equation}
The fact that $A$ is diagonal allows us to use the local law in order to bound 
\begin{equation}
|\langle (G_2 - m_2) AN \dg\tilde{S}_{\mu} \rangle| \prec \frac{1}{\sqrt{N\eta_2}}.
\end{equation}
Additionally, $\langle AN \dg\tilde{S}_{\mu} \rangle$ is $\mathcal{O}(1)$, while $|\langle G_1 G_2 N \dg\tilde{S}_{\mu} \rangle| \prec \frac{\Lambda_k}{L \sqrt{ \eta_*}}$. All these estimates together complete the proof of \eqref{eq:GGAdiag}. 

Other bounds in Lemma \ref{GGAdiag} are proved similarly. For example, to bound $\Ex{G_1 \Im G_2 A}$ we use the identity
\begin{equation}
\begin{aligned}
    \Ex{G_1 \Im G_2 A} &= m_1 \Ex{\Im G_2A} - m_1 \Ex{\underline{WG_1\Im G_2A}} \\
    &+ m_1 \frac{1}{N}\sum_{i,j =1}^N S_{ij}(G_1 \Im G_2)_{jj}(G_2^* A)_{ii} \\
    &+ m_1 \frac{1}{N}\sum_{i,j =1}^N S_{ij}(G_1 G_2)_{jj}(\Im G_2 A)_{ii} \\
    &+ m_1 \frac{1}{N}\sum_{i,j =1}^N S_{ij} ((G_1)_{jj} - m_1) (G_1\Im G_2A)_{ii} \\
\end{aligned}
\end{equation}
After splitting $S$, we get
\begin{equation}\label{eq:GImGAdiag-identity}
\begin{aligned}
    \Ex{G_1 \Im G_2 A} &= m_1 \Ex{\Im G_2A} - m_1 \Ex{\underline{WG_1\Im G_2A}} \\
    &+ m_1 \frac{1}{N}\sum_{\mu =1}^N \Ex{G_1 \Im G_2 N\dg\tilde{S}_\mu^\circ} \Ex{G_2^*AN\dg\tilde{S}_\mu} + m_1 \Ex{G_1 \Im G_2} \Ex{G_2^* A} \\
    &+ m_1 \frac{1}{N}\sum_{\mu =1}^N \Ex{G_1 G_2 N\dg\tilde{S}_\mu^\circ} \Ex{\Im G_2AN\dg\tilde{S}_\mu} + m_1 \Ex{G_1G_2} \Ex{\Im G_2 A} \\
    &+ m_1 \frac{1}{N}\sum_{\mu =1}^N \Ex{G_1 N\dg\tilde{S}_\mu^\circ} \Ex{G_1 \Im G_2 A N\dg\tilde{S}_\mu} + m_1 \Ex{G_1 - m_1} \Ex{G_1 \Im G_2 A} 
\end{aligned}
\end{equation}

We use \eqref{eq:GGAdiag} to get 
\begin{equation}
\begin{aligned}
    &\left|\Ex{G_1 G_2 N\dg\tilde{S}_\mu^\circ} \Ex{\Im G_2AN\dg\tilde{S}_\mu}\right| \\
    &\prec \frac{\Lambda_k}{L\sqrt{ \eta_*}}\cdot \left|\Ex{(\Im G_2-\Im m_2)AN\dg\tilde{S}_\mu} + \Im m_2 \Ex{AN\dg\tilde{S}_\mu}\right| \prec \frac{\Lambda_k\rho_2}{L\sqrt{ \eta_*}}.
\end{aligned}
\end{equation}
Using the bounds \eqref{eq:GGAbound-term1}, \eqref{eq:GGAbound-term2}, \eqref{eq:GGAbound-term3} and Lemma \ref{lem:underline} we get the same self-consistent equation for $\Ex{G_1 \Im G_2 A}$ as \eqref{eq:G1G2S-system} with error term $\mathcal{O}\left(\frac{\rho_2\Lambda_1}{L\sqrt{\eta_*}}\right)$ on the right. By inverting $I-C$ the same way we get
\begin{equation}
    \left|\Ex{G_1 \Im G_2 N\dg\tilde{S}^\circ_\mu}\right| \prec \frac{\rho_2\Lambda_1}{L\sqrt{\eta_*}}.
\end{equation}
By plugging this into \eqref{eq:GImGAdiag-identity}, we get \eqref{eq:GImGAdiag}.

The bound \eqref{eq:ImGImGAdiag} is proved similarly.


\end{proof}

\begin{remark}\label{imaginary-rho}
As one can see from the proof above, the computation of the traces involving imaginary parts of one the Green's functions matrices involve more terms, but these terms can be analyzed in a manner that is very similar to those traces that do not involves the imaginary part. The most important point to realize is that the inclusion of the imaginary part causes the appearance of an extra factor of $\rho$. In most cases, this either uses the fact that $\Ex{\Im G} = O(\rho)$ or that $\frac{1}{N\eta} = \frac{\rho}{L}$.

\end{remark}

\subsection{Bounds on $\Ex{GAGA}$}

\begin{lemma}\label{lem:GAGAdiag}
For $A_1, A_2 \in \mathbb{M}_k^\circ$ we have
\begin{equation}
   | \Ex{G_1 A_1 G_2 A_2} |
   \prec 1 + \frac{\Lambda_k^2}{\sqrt{L}} +  \frac{\Lambda_1 \Lambda_k \Lambda_{k+1}}{\sqrt{NL}},
\end{equation}

\begin{equation}
   | \Ex{G_1 A_1 \Im G_2 A_2} |\prec \rho_2 \left[ 1 + \frac{\Lambda_k^2}{\sqrt{L}} +  \frac{\Lambda_1 \Lambda_k \Lambda_{k+1}}{\sqrt{NL}}\right],
\end{equation}
and
\begin{equation}
   | \Ex{\Im G_1 A_1 \Im G_2 A_2} |\prec \rho_1 \rho_2 \left[ 1 + \frac{\Lambda_k^2}{\sqrt{L}} +  \frac{\Lambda_1 \Lambda_k \Lambda_{k+1}}{\sqrt{NL}}\right]. 
\end{equation}

\end{lemma}

\begin{proof}
We prove the first inequality here. The other two are proved similarly. See Remark \ref{imaginary-rho} for details.

First, we use the identity
\begin{equation}
\begin{aligned}
    \Ex{G_1 A_1 G_2 A_2}  &= m_1 m_2 \Ex{A_1 A_2} + m_1 \Ex{A_1(G_2 - m_2 I)A_2} - m_1 \Ex{\underline{ W G_1 A_1 G_2 A_2}} \\
    &+ \frac{1}{N} m_1 \sum_{i, j} S_{ij} (G_1-m_1 I)_{jj} (G_1 A_1 G_2 A_2)_{ii} \\
    & + \frac{1}{N} m_1 \sum_{i, j} S_{ij} (G_1 A_1 G_2)_{jj} (G_2 A_2)_{ii}.
\end{aligned}
\end{equation}


By splitting the terms on the last two lines, we get,
\begin{equation}\label{eq:G1A1G2A2diag}
\begin{aligned}
    \Ex{G_1 A_1 G_2 A_2} &= m_1 m_2 \Ex{A_1 A_2} + m_1 \Ex{A_1(G_2 - m_2 I)A_2} - m_1 \Ex{\underline{ W G_1 A_1 G_2 A_2}} \\
    &+\frac{m_1}{N}\sum_\mu  \Ex{ N \dg\tilde{S}^{\circ}_{\mu} G_1} \Ex{G_1 A_1 G_2 A_2 N\dg \tilde{S}_\mu} + m_1 \Ex{G_1 - m_1} \Ex{G_1A_1G_2A_2} \\
    &+\frac{m_1}{N} \sum_{\mu} \Ex{G_1 A_1 G_2 N\dg\tilde{S}_{\mu}^\circ} \Ex{G_2 A_2 N \dg\tilde{S}_{\mu}} + m_1 \Ex{G_1 A_1 G_2} \Ex{G_2 A_2}.
\end{aligned}
\end{equation}

Now we plug in $A_2 = N\dg{\tilde{S}^\circ_\nu}$ into the identity above and get a system of equations.

\begin{equation}\label{eq:G1A1G2S}
\begin{aligned}
    \Ex{G_1 A_1 G_2 N\dg{\tilde{S}^\circ_\nu}} &= m_1 m_2 \Ex{A_1 N\dg{\tilde{S}^\circ_\nu}} + m_1 \Ex{A_1(G_2 - m_2 I)N\dg{\tilde{S}^\circ_\nu}} - m_1 \Ex{\underline{ W G_1 A_1 G_2 N\dg{\tilde{S}^\circ_\nu}}} \\
    &+ \frac{m_1}{N}\sum_\mu  \Ex{ N \dg\tilde{S}^{\circ}_{\mu} G_1} \Ex{G_1 A_1 G_2 N\dg{\tilde{S}^\circ_\nu} N\dg \tilde{S}_\mu} + m_1 \Ex{G_1 A_1 G_2} \Ex{G_2 A_2} \\
    &+ \sum_{\mu} C_{\nu \mu} \Ex{G_1 A_1 G_2 N\dg\tilde{S}_{\mu}^\circ},
\end{aligned}
\end{equation}
where
\begin{equation}
\begin{aligned}
    C_{\nu \mu} &= m_1\Ex{G_1 - m_1} \delta_{\nu \mu} + m_1 \frac{1}{N} \Ex{G_2 N\dg{\tilde{S}^{\circ}_{\nu}} N\dg{\tilde{S}_{\mu}} } \\
    &= m_1 m_2 \left(S_{\nu \mu} - \frac{1}{N}\right) + \mathcal{O}\left(\frac{1}{N^{3/2} \eta_*^{1/2}}\right) + \mathcal{O}\left(\frac{1}{N\eta_*}\right)\delta_{\nu \mu}.
\end{aligned}
\end{equation}

To bound the error terms in \eqref{eq:G1A1G2S} we need the following lemma.
\begin{lemma}
If $A \in \mathbb M_k^\circ$ and $B\in \mathbb M_l$, then 
\begin{equation}
\begin{aligned}
     &\left|\langle G_1AG_2B \rangle\right| \prec \Lambda_k\Lambda_l +   \frac{\Lambda_k}{ L  \eta_*},\,\\
     &\left|\langle \Im G_1AG_2B \rangle\right| \prec \rho_1 \Lambda_k\Lambda_l +   \frac{\rho_1 \Lambda_k}{ L  \eta_*},\,\\
     &\left|\langle \Im G_1A \Im G_2B \rangle\right| \prec \rho_1 \rho_2 \Lambda_k\Lambda_l +   \frac{\rho_1 \rho_2 \Lambda_k}{ L  \eta_*}.\,
\end{aligned}
\end{equation}
\end{lemma}
\begin{proof}
Let us divide the matrix $B = B^\circ  + \langle B \rangle I$, where $B^\circ \in  \mathbb M_l^\circ$ by the definition. Then, we have that
\begin{equation}
    \langle G_1AG_2B \rangle =  \langle G_1AG_2B^\circ \rangle +  \langle B \rangle \langle G_2 G_1A\rangle\,.
\end{equation}
The desired result follows from \cite[(5.34)]{cipolloni-erdos-schroder-2021} and Lemma \ref{GGAdiag}.
\end{proof}

Then 
\begin{equation}\label{eq:GAGAerror1d}
\left|\Ex{ N \dg\tilde{S}^{\circ}_{\mu} G_1} \Ex{G_1 A_1 G_2 A_2 N\dg(\tilde{S}_\mu)}\right| \prec  \frac{\Lambda_1}{\sqrt{L N}} \left( \Lambda_k \Lambda_{k+1} + \frac{\Lambda_k}{L \sqrt{\eta_*}}\right), 
\end{equation}
and
\begin{equation}\label{eq:GAGAerror2d}
    \left|\Ex{G_1 A_1 G_2} \Ex{G_2 A_2} \right| \prec \frac{\Lambda_k^2}{L^{3/2} \sqrt{N\eta_*}} \prec \frac{\Lambda_k^2}{L^2}. 
\end{equation}
From Lemma \ref{lem:underline}, we have
\begin{equation}
    \left| \Ex{\underline{ W G_1 A_1 G_2 N\dg{\tilde{S}^\circ_\nu}}}\right| \prec \frac{\Lambda_k^2}{\sqrt{L}}.
\end{equation}

Then by using the local law for $G_2 - m_2$, we have in general for diagonal $A_1$ and $A_2$,
\begin{equation} \label{eq:AG-mA}
\left| m_1 \Ex{A_1(G_2 - m_2 I) A_2} \right|  \prec \frac{1}{\sqrt{N \eta_2}} \prec \frac{1}{\sqrt{L}}.
\end{equation}
This crucially used the fact that $A_1$ and $A_2$ are diagonal to get a simpler estimate. We specialize this estimate in the case that $A_2 = N \dg(\tilde{{S}^{\circ}})$. Recall that we use $\dg(\tilde{S}^{\circ})$ to denote the vector constructed by considering each $\dg(\tilde{S^{\circ}_{\mu}})$.

Substituting all of these estimates in \eqref{eq:G1A1G2S}, we get
\begin{equation}
\begin{aligned}
    \left(I - C \right)\Ex{G_1 A_1 G_2 N\dg{\tilde{S}^\circ}} &= m_1 m_2 \Ex{A_1 N\dg{\tilde{S}^\circ}}  \\
    &+ \mathcal{O}\left(\frac{1}{\sqrt{L}} + \frac{\Lambda_k^2}{\sqrt{L}} + \frac{\Lambda_1 \Lambda_2 \Lambda_k}{\sqrt{NL}} + \frac{\Lambda_1 \Lambda_k}{L^2}\right)\,.
\end{aligned}
\end{equation}

By inverting matrix $I - C$ using a similar argument to the proof of Lemma \ref{lem:bndGA}, we get
\begin{equation}\label{eq:GAGSerror1}
  \Big|  \Ex{G_1 A_1 G_2 N\dg{\tilde{S}^\circ}}\Big| \prec 1  + \frac{\Lambda_k^2}{\sqrt{L}} + \frac{\Lambda_1 \Lambda_2 \Lambda_k}{\sqrt{NL}} .
\end{equation}

Now we bound $\Ex{G_1A_1G_2A_2}$. To bound $\Ex{G_2 A_2 N \dg\tilde{S}_\mu}$ we write $G_2 = (G_2 - m_2) + m_2$. Then since $A_2$ is diagonal, we have
\begin{equation}
    \left|\Ex{G_2 A_2 N \dg\tilde{S}_\mu}\right| \prec 1 + \frac{1}{\sqrt{N\eta_2}} \prec 1.
\end{equation}
Then by using \eqref{eq:GAGSerror1}, we get
\begin{equation}\label{eq:GAGAbigerror}
    \left|\frac{m_1}{N} \sum_{\mu} \Ex{G_1 A_1 G_2 N\dg\tilde{S}_{\mu}} \Ex{G_2 A_2 N \dg\tilde{S}_{\mu}}\right| \prec 1 + \frac{\Lambda_k^2}{\sqrt{L}} + \frac{\Lambda_1 \Lambda_2 \Lambda_k}{\sqrt{NL}}.
\end{equation}
Finally, by plugging in \eqref{eq:GAGAbigerror}, \eqref{eq:GAGAerror1d}, \eqref{eq:GAGAerror2d} and Lemma \ref{lem:underline} into \eqref{eq:G1A1G2A2diag} and moving $\Ex{G_2 - m_2} \Ex{G_1A_1G_2A_2}$ term to the left, we get
\begin{equation}
\begin{aligned}
    \left|\Ex{G_1 A_1 G_2 A_2}\right| &\prec 1 + \frac{\Lambda_k^2}{\sqrt{L}} +  \frac{\Lambda_1 \Lambda_k \Lambda_{k+1}}{\sqrt{NL}} .
\end{aligned}
\end{equation}

\end{proof}

\subsection{Bounds on $\Ex{\Im GA\Im G^tA}$}
\begin{lemma}
\label{GAGtA}
If $A_1,A_2 \in \mathbb M_k$, then 
\begin{equation}
  \left|\Ex{\Im G_1 A_1 \Im G_2^t A_2}\right| \prec \rho_1\rho_2\left(1 + \frac{\Lambda_1 \Lambda_{k+1}^2 }{\sqrt{NL}} +  \frac{\Lambda_{k+1}^2}{L}\right)\,.
\end{equation}
\end{lemma}
\begin{proof}This case can be proved by estimating each term separately. First, write
\begin{equation}
\label{eq:Gtcase}
\begin{split}
    \Ex{\Im G_1 A_1 \Im G_2^t A_2}  
    &= \Im m_1 \Im m_2 \Ex{A_1 A_2} + \Im m_1 \Ex{A_1(\Im G_2^t - \Im m_2 I)A_2}\\& - m_1 \Ex{\underline{ W \Im G_1 A_1 \Im G_2^t A_2}} - \Im m_1 \Ex{\underline{W G_1^* A_1 \Im G_2^t A_2}} \\
    &+\frac{m_1}{N}\sum_\mu  \Ex{ N \dg\tilde{S}^{\circ}_{\mu} G_1} \Ex{\Im G_1 A_1 \Im G_2^t A_2 N\dg(\tilde{S}_\mu)} \\ & + \frac{m_1}{N} \sum_{\mu}  \Ex{ N \dg\tilde{S}^{\circ}_{\mu} \Im G_1} \Ex{G_1^* A_1 \Im G_2^t A_2 N\dg(\tilde{S}_\mu)}\\
    &+ \frac{\Im m_1}{N}\sum_\mu  \Ex{ N \dg\tilde{S}^{\circ}_{\mu} G_1^*} \Ex{ G_1^* A_1 \Im G_2^t A_2 N\dg(\tilde{S}_\mu)}\\
    &+ m_1 \Ex{G_1 - m_1} \Ex{\Im G_1A_1\Im G_2^tA_2}+ m_1 \Ex{\Im G_1 - \Im m_1} \Ex{G_1^* A_1\Im G_2^tA_2} \\
    &+ \Im m_1 \Ex{ G_1^* - \overline{m_1}} \Ex{G_1^* A_1\Im G_2^tA_2}\\
&+\frac{m_1}{N^2}\sum_\mu  \Ex{\Im G_1A_1 G^t_2 N\dg(\tilde{S}_\mu)A_2^t \Im G_2 N\dg(\tilde{S}_\mu)}\\&+\frac{m_1}{N^2}\sum_\mu \Ex{\Im G_1A_1\Im G^t_2 N\dg(\tilde{S}_\mu)A_2^t G_2^* N\dg(\tilde{S}_\mu)}\\
&+\frac{\Im m_1}{N^2}\sum_\mu  \Ex{ G_1^*A_1 G^t_2 N\dg(\tilde{S}_\mu)A_2^t \Im G_2 N\dg(\tilde{S}_\mu)}\\&+\frac{\Im m_1}{N^2}\sum_\mu \Ex{G_1^* A_1\Im G^t_2 N\dg(\tilde{S}_\mu)A_2^t G_2^* N\dg(\tilde{S}_\mu)}.
\end{split}
\end{equation}
Using Lemma \ref{lem:bndGA}, \cite[(5.34),(5.35)]{cipolloni-erdos-schroder-2021}, and \eqref{eq:stronglocallaw}, we get
\begin{equation*}
    \begin{split}
    &\left|\Ex{N\dg\tilde{S}^\circ_\mu G_1}\Ex{\Im G_1A_1\Im G_2^tA_2N\dg(\tilde{S}_\mu)}\right| \prec \frac{\sqrt{\rho_1}\Lambda_1}{N\sqrt{\eta_1}}\cdot \rho_1\rho_2\Lambda_{k+1}^2 \prec \frac{\rho_1\rho_2\Lambda_1\Lambda_{k+1}^2}{\sqrt{NL}}\,,\\
    &  \left|  \Ex{\Im G_1A_1 G^t_2 N\dg(\tilde{S}_\mu)A_2^t \Im G_2 N\dg(\tilde{S}_\mu)}\right| \prec  \frac{\sqrt{\rho_1\rho_2}\Lambda_k\Lambda_{k+1}}{\sqrt{\eta_1\eta_2}} \prec \frac{\rho_1\rho_2\Lambda_{k+1}^2}{L},\,\\
    &|\Im m_1\Ex{A_1(\Im G_2^t - \Im m_2 I) A_2}|  \prec \frac{\rho_1 \rho_2}{\sqrt{L}}.
    \end{split}
\end{equation*}
Other terms are estimated similarly. Substituting these bounds into \eqref{eq:Gtcase} gives the desired result.


\end{proof}

\subsection{Continuity argument for bounding $\Lambda_1$}
For each value of $E$ and $J$ there is a unique value of $\eta$ such that $N \eta \rho(E + i\eta) = J$. We let $F(E,J)$ be the unique $\eta$ so that this is true. 

We can now define the functions
\begin{equation}
\begin{aligned}
    G_{A}(J)&:= \max_{E_1 ,E_2 \in \{\gamma_a: 1\leq a \leq N\}} \frac{\Ex{\Im G(E_1 + \ti F(E_1,J)) A \Im G(E_2 + \ti F(E_2,J)) A^*}}{\rho_1(E_1+ \ti F(E_1,J)) \rho_2(E_2 + \ti F(E_2,J))}, \\
    G^t_{A}(J)&:= \max_{E_1 ,E_2 \in \{\gamma_a: 1\leq a \leq N\}} \frac{\Ex{\Im G(E_1 + \ti F(E_1,J)) A \Im G(E_2 + \ti F(E_2,J))^t A^*}}{\rho_1(E_1+ \ti F(E_1,J)) \rho_2(E_2 + \ti F(E_2,J))}.
\end{aligned}
\end{equation}

\begin{lemma}

Uniformly in $E_1,E_2$ that there is a constant $C$ so that,
\begin{equation} \label{eq:Gderbound}
\begin{aligned}
   \left| \partial_J\frac{\Ex{\Im G(E_1 + \ti F(E_1,J)) A \Im G(E_2 + \ti F(E_2,J)) A^*}}{\rho_1(E_1+ \ti F(E_1,J)) \rho_2(E_2 + \ti F(E_2,J))} \right| &\le N^{C}, \\
   \left| \partial_J\frac{\Ex{\Im G(E_1 + \ti F(E_1,J)) A \Im G(E_2 + \ti F(E_2,J))^t A^*}}{\rho_1(E_1+ \ti F(E_1,J)) \rho_2(E_2 + \ti F(E_2,J))} \right| &\le N^{C}.
\end{aligned}
\end{equation}

\end{lemma}
\begin{proof}


First, let us find $\partial_J \eta$ at $E$. We have,
\begin{equation}
\begin{aligned}
    & 1 = N \partial_{J}\eta (\rho + \eta \partial_{\eta} \rho),\\
    & \partial_J \eta = \frac{1}{N (\rho + \eta \partial_{\eta} \rho)}.
\end{aligned}
\end{equation}

Now, if $J \ge N^{\epsilon}$, using the fact that $\rho \le 1$ implies $\eta \ge N^{-1 + \epsilon}$. Furthermore, we would also know that $|\rho| \gtrsim \eta$ for $E \in [-2,2]$. 

We have the following integral expression for $\rho$.
\begin{equation}
\begin{aligned}
    &\rho(E+ \ti \eta) =\int_{-2}^2 \frac{ \eta \rho_{sc}(x) }{(x-E)^2 + \eta^2} \text{d} x,\\
    &\partial_{\eta} \rho(E+ \ti \eta) = \int_{-2}^2 \frac{\rho_{sc}(x) ( (x-E)^2 - \eta^2)}{((x-E)^2 +\eta^2)^2} \text{d} x,\\
    & \rho + \eta \partial_{\eta} \rho = \int_{-2}^2 \frac{2 \eta \rho_{sc}(x) (x-E)^2 }{(( x-E)^2 + \eta^2)^2} \text{d} x \gtrsim \eta .
\end{aligned}
\end{equation}
Thus, $|\partial_J \eta| \le N^{-\epsilon}$. 

The above expression should also allow us to assert that $|\partial_\eta \rho| \le \frac{1}{\eta^4} \le N^4$.

\begin{equation}
    \max_{i,j}|\partial_\eta G_{ij}| =  \max_{i,j}\left|\sum_\alpha \partial_\eta \left(\frac{1}{\lambda_\alpha - z} \right) u_\alpha(i) \overline u_\alpha(j)\right| \lesssim N^3\,.
\end{equation}

By applying the product rule, this would imply equation \eqref{eq:Gderbound}.
\end{proof}

\begin{corollary}
For any matrix $A$ with $||A|| \le 1$ and
for any value of  $N^{\epsilon} \le J \le N$, we have that
\begin{equation}
\begin{aligned}
    & G_{A}(J- N^{-C -2}) \le G_{A}(J) + N^{-2},\\
    & G_A^t(J - N^{-C-2}) \le G_{A}^t(J) + N^{-2}.
\end{aligned}
\end{equation}
Taking the union over $M \in \mathbb{M}_1$, which is a $O(N)$ family of matrices, and applying Lemma \ref{GAGA&Xi}, we also have,
\begin{equation}
\label{eq:LamCon}
    \Lambda_1(J - N^{-C-2}) \prec \Lambda_1(J) +  N^{-2}
\end{equation}

\end{corollary}


\begin{proof}[Proof of Theorem \ref{thm:Xicontrolparam} for Diagonal Matrices]

By using Lemmas \ref{GAGA&Xi}, \ref{GAGtA}, and \ref{lem:GAGAdiag}, we would be able to derive the following relation:
$$
\Lambda_k^2(J) \prec 1 + \frac{\Lambda_1 (J)\Lambda_{k+1}^2(J)}{\sqrt{J}}\,.
$$
That is, for any $\delta, D>0$, there exists $N_0(\delta,D)>0$ such that for all $N \geq N_0$,
\begin{equation}
 \mathbb P \left(   \Lambda_k^2(J) \leq N^{\delta}\left( 1 + \frac{\Lambda_1 (J)\Lambda_{k+1}^2(J)}{\sqrt{J}}\right) \right) \geq 1 - N^{-D}\,,
\end{equation}
for all $M \in \mathbb M_k$. 

Let $\Omega_\delta$ be the event that
$$  \Lambda_k^2(J) \leq N^{\delta}\left( 1 + \frac{\Lambda_1 (J)\Lambda_{k+1}^2(J)}{\sqrt{J}}\right),\quad \Lambda_1(J - N^{-C-2}) \leq N^{\delta} \left(\Lambda_1(J) +  N^{-2} \right)
$$
holds for $k=1,...,\lceil 4/\epsilon \rceil$, $J= N - t N^{-C+2}$, $t \in \{0, \ldots, \lfloor N^{C-2} -  N^{C-2 + \epsilon} \rfloor\}$ and all $M \in \mathbb M_1$. Then there exists $N_0(\delta,D)>0$ such that for all $N \geq N_0$, 
$\mathbb P(\Omega_\delta)\geq 1- N^{-D}$.

This identity can be iterated to show that for  $T = \lceil 4/ \epsilon \rceil$,
\begin{equation}
    \Lambda_1^2(J) \leq N^{\delta T} \left[1+ \sum_{t=1}^{T}\frac{(\Lambda_1^2(J))^t}{J^{t/2}} + \frac{(\Lambda_1^2(J))^{T+1}}{J^{(T+1)/2}} \Lambda_{T+1}^2(J)\right].
\end{equation}
on $\Omega_\delta$.

$\Lambda_{T+1}$ can be given the trivial bound $\eta_*^{-1}$, so this ultimately gives us,
\begin{equation}
    \Lambda_1^2(J) \leq  N^{\delta T} \left[  1+ \sum_{t=1}^{T}\frac{(\Lambda_1^2(J))^t}{J^{t/2}} + \frac{(\Lambda_1^2(J))^{T+1}}{J^{(T+1)/2}} \frac{1}{\eta_*^2}\right].
\end{equation}
Since $J \geq N^{ \epsilon}$, if we choose $T = \lceil 8/ \epsilon \rceil$,
this implies that either
$\Lambda_1^2(J) \leq (T+2) N^{\delta T}$ or $\Lambda_1^2(J) \ge N^{-\epsilon/4}\sqrt{J} \ge N^{\epsilon/4}.$ Now we choose $\delta < \frac{\epsilon^2}{40}$.


Now, assume by induction for some $t'$, we know that $\Lambda(N - t' N^{-C +2}) \leq (T+2) N^{\delta T} $. On the event $\Omega_\delta$, we can assert that $$\Lambda_1^2(N - (t'+1) N^{-C+2}) \leq N^{\delta} ( (T+2) N^{\delta T}  + N^{-2}) \leq (T+3)N^{\delta (T+1)} < N^{\epsilon/4}\,. $$  Hence, we must have $\Lambda_1^2(N - (t'+1) N^{-C+2}) \leq  (T+2) N^{\delta T} $ as well, due to the dichotomy that we have shown earlier. By induction, on $\Omega_\delta$ we have
$$\Lambda_1(N^{\epsilon}) \leq (T+2) N^{\delta T}\,.$$
\end{proof}

\section{Proof of Theorem \ref{thm:Xicontrolparam} for general $M$}\label{sec:generalM}

\subsection{Bounds on $\Ex{GA}$}

\begin{lemma} \label{lem:bndGAgeneral}
Let $A \in \mathbb M_k^\circ$. Then, we have that,
\begin{equation}
\label{eq:GA-nondia}
   | \langle GA \rangle| \prec \rho\left(\frac{ \Lambda_k}{\sqrt{NL}} + \frac{\Lambda_{k+3}}{NL}\right). 
\end{equation}
\end{lemma}

\begin{proof}



At this point, we can return to an analysis for general matrices $A$. From the equation \eqref{eq:longGA}, we see that,
\begin{equation}
     \langle GA \rangle [ 1- \langle G - mI \rangle] = - m\langle{\underline{WGA}}\rangle + \frac{1}{N} \sum_{\mu} \langle G B_\mu\rangle \langle G N \dg(\tilde{S}_{\mu}^{\circ}) \rangle,
\end{equation}
with $B_\mu:= A (N \dg\tilde{S}_{\mu})$. Using
$$  \Ex{GB_\mu} = m\Ex{B_\mu} + \Ex{B_\mu}\Ex{G-m}+ \Ex{GB_\mu^\circ},$$
and
$$|\langle G N \dg(\tilde{S}_{\mu}^{\circ}) \rangle | \prec \frac{\rho}{\sqrt{NL}}\,,\quad \Ex{B_\mu} \prec 1, \quad |\Ex{G-m}| \prec \frac{1}{N\eta}, \quad |\langle{\underline{WGA}}\rangle| \prec  \frac{\rho\Lambda_k}{\sqrt{NL}}  \,,$$
we get
\begin{equation}
     |\langle GA \rangle| \prec \frac{\rho}{\sqrt{NL}} \left(\Lambda_k+\max_\mu\Ex{G B^\circ_\mu} \right)\,.
\end{equation}
By iterating this bound, we get
\begin{equation}
    \sup_{A \in \mathbb M_k^\circ}  |\langle GA \rangle| \prec \sum_{t=0}^T  \left(\frac{\rho}{\sqrt{NL}}\right)^{t+1} \Lambda_{k+t} + \left(\frac{\rho}{\sqrt{NL}}\right)^{T+1} \Lambda_{k+T+1}\,.
\end{equation}
If we take $T = 3$ and use the trivial bounds $\Lambda_{k + T + 1}\prec \frac{1}{\eta}$, $\Lambda_k \geq 1$, we get 
\begin{equation}
    |\langle GA \rangle| \prec \sum_{t \geq 0}^{3}\frac{\Lambda_{k+t} \rho^{t+1}}{\sqrt{NL}^{t+1}} \prec \rho \frac{\Lambda_k}{\sqrt{NL}} + \rho\frac{\Lambda_{k+3}}{NL}.
\end{equation}

\end{proof}

\subsection{Bounds on $\Ex{G G A}$}



\begin{lemma}
\label{GGA}
Let $A \in \mathbb M_k^\circ$. Then
\begin{equation}
\begin{aligned}
    &|\Ex{G_1G_2A}| \prec \frac{\Lambda_k}{L \sqrt{\eta_*}} + \frac{\Lambda_{k+4}}{L^2},\\
    &|\Ex{\Im G_1 G_2 A} | \prec \rho_1 \frac{\Lambda_k}{L \sqrt{\eta_*}} + \rho_1 \frac{\Lambda_{k+4}}{L^2},\\
    & |\Ex{\Im G_1 \Im G_2 A}| \prec \rho_1 \rho_2 \frac{\Lambda_k}{L \sqrt{\eta_*}} + \rho_1 \rho_2 \frac{\Lambda_{k+4}}{L^2}. 
\end{aligned}
\end{equation}
\end{lemma}

\begin{proof}
We prove the first inequality here. The other two are proved similarly. See Remark \ref{imaginary-rho} for details.

From \eqref{eq:G1G2A} we get
\begin{equation*}
\begin{aligned}
    \Ex{G_1G_2A}[1 - m_1\Ex{G_1 - m_1}] &= m_1 \Ex{G_2A} - m_1 \Ex{\underline{WG_1G_2A}} \\
    &+ \frac{1}{N} m_1 \sum_{\mu = 1}^N \Ex{G_1G_2N\dg{\tilde{S}^\circ_\mu}} \Ex{G_2AN\dg{\tilde{S}_\mu}} + m_1 \Ex{G_1G_2} \Ex{G_2A} \\
    &+ \frac{1}{N} m_1 \sum_{\mu = 1}^N \Ex{G_1 N\dg{\tilde{S}_{\mu}^\circ}} \Ex{G_1G_2AN\dg{\tilde{S}_{\mu}}}\,.
\end{aligned}
\end{equation*}
Suppose $B = AN\dg{\tilde{S}_{\mu}}$ or $B = I$. We apply Cauchy-Schwarz to bound $\Ex{G_1G_2B}$.
\begin{equation}
    \left|\Ex{G_1G_2B}\right| \le \Ex{G_1 G_1^*}^{\frac{1}{2}} \Ex{G_2 B B^* G_2^*}^{\frac{1}{2}} \prec \frac{\sqrt{\rho_1\rho_2}}{\sqrt{\eta_1\eta_2}}.
\end{equation}
Using this estimate, the estimate for $\Ex{GA}$ from above and $\Ex{G_1 - m_1} \prec \frac{1}{N\eta_1}$, we get
\begin{equation}
\begin{aligned}
    \left(1 + \mathcal{O}\left(\frac{1}{N\eta_1}\right)\right)\Ex{G_1G_2A} &= - m_1 \Ex{\underline{WG_1G_2A}} \\
    &+  \frac{1}{N} m_1 \sum_{\mu = 1}^N \Ex{G_1G_2N\dg{\tilde{S}^\circ_\mu}} \Ex{G_2AN\dg{\tilde{S}_\mu}}
    \\&+ \mathcal O\left( \frac{\Lambda_{k}}{L\sqrt{\eta_*}} + \frac{\Lambda_{k+3}}{L^2}  \right).
\end{aligned}
\end{equation}
By Lemma \ref{lem:bndGAgeneral},
\begin{equation}
   \left| \Ex{G_2AN\dg{\tilde{S}_\mu}} \right|  \prec 1 +\rho_2 \frac{\Lambda_{k+1}}{\sqrt{NL}} +\rho_2 \frac{\Lambda_{k+4}}{N L} \prec 1 +\rho_2 \frac{\Lambda_{k+4}}{\sqrt{NL}}.
\end{equation}
Hence we get
\begin{align*}
    |\Ex{G_1G_2A}| \prec  \frac{\Lambda_k}{L\sqrt{\eta_*}} + \frac{1}{L\sqrt{\eta_*}}\left(1 + \frac{\Lambda_{k+4}}{\sqrt{NL}}\right) + \frac{\Lambda_{k+3}}{L^2} \prec \frac{\Lambda_k}{L \sqrt{\eta_*}} + \frac{\Lambda_{k+4}}{L^2}.
\end{align*}



\end{proof}
\subsection{Bounds on $\Ex{GAGA}$}

\begin{lemma}\label{lem:GAGAgeneral}
Suppose $M$ is a traceless matrix with $\|M\| \prec 1$. For $A_1, A_2 \in \mathbb{M}_k^\circ$, we have
\begin{equation}
\begin{aligned}
    |\Ex{G_1 A_1 G_2 A_2} | & \prec 1 + \frac{\Lambda_{k+4}^2}{\sqrt{L}}, \\
    |\Ex{\Im G_1 A_1 G_2 A_2} | & \prec \rho_1\left(1 + \frac{\Lambda_{k+4}^2}{\sqrt{L}}\right), \\
    |\Ex{\Im G_1 A_1 \Im G_2 A_2} | & \prec \rho_1\rho_2\left(1 + \frac{\Lambda_{k+4}^2}{\sqrt{L}}\right).
\end{aligned}
\end{equation}

\end{lemma}
\begin{proof}
We prove the first inequality here. The other two are proved similarly. See Remark \ref{imaginary-rho} for details.

The proof is similar to the proof of Lemma \ref{lem:GAGAdiag}. The only difference is the size of the error terms.

Similarly to \eqref{eq:G1A1G2A2diag}, we have
\begin{equation}\label{eq:G1A1G2A2general}
\begin{aligned}
    \Ex{G_1 A_1 G_2 A_2} &= m_1 m_2 \Ex{A_1 A_2} + m_1 \Ex{A_1(G_2 - m_2 I)A_2} - m_1 \Ex{\underline{ W G_1 A_1 G_2 A_2}} \\
    &+\frac{m_1}{N}\sum_\mu  \Ex{ N \dg\tilde{S}^{\circ}_{\mu} G_1} \Ex{G_1 A_1 G_2 A_2 N\dg \tilde{S}_\mu} + m_1\Ex{G_1 - m_1} \Ex{G_1A_1G_2A_2} \\
    &+\frac{m_1}{N} \sum_{\mu} \Ex{G_1 A_1 G_2 N\dg\tilde{S}_{\mu}^\circ} \Ex{G_2 A_2 N \dg\tilde{S}_{\mu}} + m_1 \Ex{G_1 A_1 G_2} \Ex{G_2 A_2}.
\end{aligned}
\end{equation}

Now we plug in $A_2 = N\dg{\tilde{S}^\circ_\nu}$ into the identity above and get a system of equations.

\begin{equation}\label{eq:G1A1G2Sgeneral}
\begin{aligned}
    \Ex{G_1 A_1 G_2 N\dg{\tilde{S}^\circ_\nu}} &= m_1 m_2 \Ex{A_1 N\dg{\tilde{S}^\circ_\nu}} + m_1 \Ex{A_1(G_2 - m_2 I)N\dg{\tilde{S}^\circ_\nu}} \\&- m_1 \Ex{\underline{ W G_1 A_1 G_2 N\dg{\tilde{S}^\circ_\nu}}} + m_1 \Ex{G_1 A_1 G_2} \Ex{G_2 N\dg{\tilde{S}^\circ_\nu}}\\
    &+ \frac{m_1}{N}\sum_\mu  \Ex{ N \dg\tilde{S}^{\circ}_{\mu} G_1} \Ex{G_1 A_1 G_2 N\dg{\tilde{S}^\circ_\nu} N\dg\tilde{S}_\mu}  \\
    &+ \sum_{\mu} C_{\nu \mu} \Ex{G_1 A_1 G_2 N\dg\tilde{S}_{\mu}^\circ},
\end{aligned}
\end{equation}
where
\begin{equation}
\begin{aligned}
    C_{\nu \mu} &= m_1\Ex{G_1 - m_1} \delta_{\nu \mu} + m_1 \frac{1}{N} \Ex{G_2 N\dg{\tilde{S}^{\circ}_{\nu}} N\dg{\tilde{S}_{\mu}} } \\
    &= m_1 m_2 \left(S_{\nu \mu} - \frac{1}{N}\right) + \mathcal{O}\left(\frac{1}{N^{3/2} \eta_*^{1/2}}\right) + \mathcal{O}\left(\frac{1}{N\eta_1}\right)\delta_{\nu \mu}.
\end{aligned}
\end{equation}

To bound the error terms in \eqref{eq:G1A1G2Sgeneral} we need the following lemma.
\begin{lemma} \label{lem:genGAGB}
If $A \in \mathbb M_k^\circ$ and $B\in \mathbb M_l$, then 
\begin{equation}
\label{eq:nondA-gagb-1}
\begin{aligned}
     &|\langle G_1AG_2B \rangle| \prec \Lambda_k\Lambda_l +  \frac{\Lambda_k}{L \sqrt{\eta_*}} + \frac{\Lambda_{k+4}}{L^2},\,\\
     &|\langle \Im G_1 A G_2 B \rangle| \prec \rho_1 \Lambda_k \Lambda_l + \rho_1 \frac{\Lambda_k}{L \sqrt{\eta_*}} +\rho_1 \frac{\Lambda_{k+4}}{L^2},\\
     &|\langle \Im G_1 A \Im G_2 B \rangle| \prec \rho_1 \rho_2 \Lambda_k \Lambda_l + \rho_1 \rho_2 \frac{\Lambda_k}{L \sqrt{\eta^*}}  + \rho_1 \rho_2 \frac{\Lambda_{k+4}}{L^2}.
\end{aligned}
\end{equation}
Furthermore, for $B = N\dg \tilde{S}_\nu^\circ N\dg \tilde{S}_\mu$, 
\begin{equation}
\label{eq:nondA-gagb-2}
\begin{aligned}
     & |\langle G_1AG_2B \rangle| \prec \Lambda_k +  \frac{\Lambda_k}{L \sqrt{\eta_*}} + \frac{\Lambda_{k+4}}{L^2},\,\\
     &|\langle \Im G_1 A G_2 B \rangle| \prec \rho_1 \Lambda_k + \rho_1 \frac{\Lambda_k}{L \sqrt{\eta_*}} + \frac{\Lambda_{k+4}}{L^2},\\
     & |\langle \Im G_1 A \Im G_2 B \rangle| \prec \rho_1 \rho_2 \Lambda_k + \rho_1 \rho_2 \frac{\Lambda_k}{L \sqrt{\eta_*}} + \rho_1 \rho_2 \frac{\Lambda_{k+4}}{L^2}.
\end{aligned}
\end{equation}
\end{lemma}
\begin{proof}
Let us divide the matrix $B = B^\circ  + \langle B \rangle I$, where $B^\circ \in  \mathbb M_l^\circ$ by the definition. Then, we have that
\begin{equation}
    \langle G_1AG_2B \rangle =  \langle G_1AG_2B^\circ \rangle +  \langle B \rangle \langle G_2 G_1A\rangle\,.
\end{equation}
Now, \eqref{eq:nondA-gagb-1} follows from \cite[(5.34)]{cipolloni-erdos-schroder-2021}, Lemma \ref{GGA}, and equation \eqref{eq:nondA-gagb-2}. 
\end{proof}

Then 
\begin{equation}\label{eq:GAGSgeneral-error1}
\left|\Ex{ N \dg\tilde{S}^{\circ}_{\mu} G_1} \Ex{G_1 A_1 G_2 N\dg \tilde{S}_\nu^\circ N\dg \tilde{S}_\mu}\right| \prec  \frac{\rho_1}{\sqrt{NL}} \left( \Lambda_k + \frac{\Lambda_k}{L \sqrt{\eta_*}} + \frac{\Lambda_{k+4}}{L^2}\right).
\end{equation}
and
\begin{equation}\label{eq:GAGSgeneral-error2}
    \left|\Ex{G_1 A_1 G_2} \Ex{G_2 N\dg \tilde{S}_\nu^\circ}\right| \prec  \left(\frac{\Lambda_k}{L \sqrt{\eta_*}} + \frac{\Lambda_{k+4}}{L^2}\right) \frac{\rho_2}{\sqrt{NL}}\,.
\end{equation}

Using Lemma \ref{lem:bndGAgeneral}, we get
\begin{equation}\label{eq:GAGAerror4}
\begin{aligned}
    \left|\Ex{A_1(G_2 - m_2 I)N \dg\tilde{S}^{\circ}_{\nu}}\right| &= \left|\Ex{G_2-m_2}\Ex{A_1N \dg\tilde{S}^{\circ}_{\nu}} + \Ex{(G_2 - m_2)(N \dg\tilde{S}^{\circ}_{\nu}A_1)^\circ} \right| \\
    &\prec \frac{1}{N\eta_2} + \frac{\rho_2\Lambda_{k+1}}{\sqrt{NL}} + \frac{\rho_2\Lambda_{k+4}}{NL} \prec  \frac{\rho_2}{L} + \frac{\rho_2\Lambda_{k+4}}{\sqrt{NL}}.
\end{aligned}
\end{equation}

From Lemma \ref{lem:underline}, we have
\begin{equation}
    \left|\Ex{\underline{ W G_1 A_1 G_2 N\dg{\tilde{S}^\circ_\nu}}}\right| \prec \frac{\Lambda_k}{\sqrt{L}}.
\end{equation}

Then from \eqref{eq:G1A1G2Sgeneral}, we get
\begin{equation} \label{eq:1-CGAGS}
\begin{aligned}
    \left(I - C \right)\Ex{G_1 A_1 G_2 N\dg{\tilde{S}^\circ}} &= m_1 m_2 \Ex{A_1 N\dg{\tilde{S}^\circ_\nu}} +  \mathcal{O}_{\prec}\bigg( \frac{1}{L} + \frac{\Lambda_k}{\sqrt{L} } + \frac{\Lambda_{k+4}}{\sqrt{NL}} \\
    &+\frac{1}{\sqrt{NL}} \left( \Lambda_k  +  \frac{\Lambda_k}{L\sqrt{\eta_*}} + \frac{\Lambda_{k+4}}{L^2}\right)\bigg) \\
    &= m_1 m_2 \Ex{A_1 N\dg{\tilde{S}^\circ_\nu}} + \mathcal{O}_{\prec}\bigg(\frac{\Lambda_{k+4}}{\sqrt{L} }\bigg).
\end{aligned}
\end{equation}

By inverting matrix $I - C$ via a similar argument found in the proof of Lemma \ref{lem:bndGA}, we get
\begin{equation}\label{eq:GAGSerror}
\begin{aligned}
    \left|\Ex{G_1 A_1 G_2 N\dg{\tilde{S}^\circ}}\right| & \prec 1+ \frac{\Lambda_{k+4}}{\sqrt{L} }.
\end{aligned}
\end{equation}

We now return to our bound of $\langle G_1A_1 G_2 A_2 \rangle $.
At this point, we can substitute our bounds in equation \eqref{eq:GAGSerror} into equation \eqref{eq:G1A1G2A2general}.

We have some other error terms to deal with in \eqref{eq:G1A1G2A2general}. By using Lemma \ref{lem:genGAGB}, we have,
\begin{equation}\label{eq:GAGAerror1}
\begin{aligned}
\left|\Ex{ N \dg\tilde{S}^{\circ}_{\mu} G_1} \Ex{G_1 A_1 G_2 A_2 N\dg \tilde{S}_\mu}\right| &\prec  \frac{\rho_1}{\sqrt{NL}} \left( \Lambda_k \Lambda_{k+1} +  \frac{\Lambda_k}{L\sqrt{\eta_*}} + \frac{\Lambda_{k+4}}{L^2}\right) \\
& \prec \frac{\Lambda_{k+4}^2}{\sqrt{L}} .
\end{aligned}
\end{equation}

We use \eqref{eq:GAGSerror} and Lemma \ref{lem:bndGAgeneral} to get
\begin{equation}\label{eq:GAGAerror2}
\begin{aligned}
    &\left|\Ex{G_1 A_1 G_2 N\dg\tilde{S}_{\mu}} \Ex{G_2 A_2 N \dg\tilde{S}_{\mu}}\right| \\
    &\prec \left|\Ex{G_1 A_1 G_2 N\dg\tilde{S}_{\mu}}\right| \left(\Ex{G_2}\left|\Ex{A_2 N \dg\tilde{S}_{\mu}}\right| + \left|\Ex{G_2 \left(A_2 N \dg\tilde{S}_{\mu}\right)^\circ}\right|\right) \\
    &\prec \left( 1 +\frac{\Lambda_{k+4}}{\sqrt{L} }\right) \left(1 + \frac{\Lambda_{k+1}}{\sqrt{NL}} + \frac{\Lambda_{k+4}}{NL}\right) \prec \left( 1 + \frac{\Lambda_{k+4}}{\sqrt{L} }\right)\left(1 +\frac{\Lambda_{k+4}}{\sqrt{NL}}\right) \\
   & \prec 1 +  \frac{\Lambda_{k+4}^2}{\sqrt{L}}.
\end{aligned}
\end{equation}

Using Lemma \ref{GGA} and Lemma \ref{lem:bndGAgeneral}, we get
\begin{equation}\label{eq:GAGAerror3}
\begin{aligned}
    \left|\Ex{G_1 A_1 G_2} \Ex{G_2 A_2}\right| & \prec  \left( \frac{\Lambda_k}{L\sqrt{\eta_*}} + \frac{\Lambda_{k+4}}{L^2} \right) \left(\frac{\Lambda_k}{\sqrt{NL}} + \frac{\Lambda_{k+3}}{N L} \right) 
    \prec \frac{\Lambda_{k+4}^2}{L^2}\,.
\end{aligned}
\end{equation}

Using Lemma \ref{lem:bndGAgeneral}, we get
\begin{equation}\label{eq:GAGAerror4}
\begin{aligned}
    \left|\Ex{A_1(G_2 - m_2 I)A_2}\right| &\le \left|\Ex{G_2-m_2}\Ex{A_1A_2}\right| + \left|\Ex{(G_2 - m_2)(A_2A_1)^\circ}\right| \\
    &\prec \frac{1}{L} + \frac{\Lambda_{k+1}}{\sqrt{LN}} + \frac{\Lambda_{k+4}}{NL} \prec \frac{1}{L} + \frac{\Lambda_{k+4}^2}{\sqrt{NL}} .
\end{aligned}
\end{equation}

From Lemma \ref{lem:underline}, we have
\begin{equation}\label{eq:GAGAerror5}
    \left|\Ex{\underline{ W G_1 A_1 G_2 A_2}}\right| \prec \frac{\Lambda_k^2}{\sqrt{L}}.
\end{equation}

Now we use bounds \eqref{eq:GAGAerror1}, \eqref{eq:GAGAerror2}, \eqref{eq:GAGAerror3}, \eqref{eq:GAGAerror4}, \eqref{eq:GAGAerror5} in \eqref{eq:G1A1G2A2general}, we have
\begin{equation}
\begin{split}
        \left[1 - \mathcal{O}_\prec\left(\frac{1}{N\eta_1}\right)\right] \left|\Ex{G_1A_1G_2A_2}\right| &\prec 1 + \frac{\Lambda_
        {k+4}^2}{\sqrt{L}} 
\end{split}
\end{equation}

We can divide by $1-\mathcal{O}_\prec\left(\frac{1}{N\eta_1} \right)$ on both sides to derive our result.
\end{proof}

\subsection{Bounds on $\Ex{\Im GA \Im G^tA}$}

\begin{lemma}
\label{lem:ImGAImGtAgeneral}
Suppose $M$ is a diagonal traceless matrix with $\|M\| \prec 1$.
If $A_1,A_2 \in \mathbb M_k$, then 
\begin{equation}
  \left|\Ex{\Im G_1 A_1 \Im G_2^t A_2}\right| \prec \rho_1\rho_2\left(1 +  \frac{\Lambda_{k+3}^2}{\sqrt{L}}\right)\,.
\end{equation}
\end{lemma}

\begin{proof} The proof is almost identical to the proof of Lemma \ref{GAGtA}. We start by using the identity 

\begin{equation}
\label{eq:Gtcase-general}
\begin{split}
    \Ex{\Im G_1 A_1 \Im G_2^t A_2}  
    &= \Im m_1 \Im m_2 \Ex{A_1 A_2} + \Im m_1 \Ex{A_1(\Im G_2^t - \Im m_2 I)A_2}\\& - m_1 \Ex{\underline{ W \Im G_1 A_1 \Im G_2^t A_2}} - \Im m_1 \Ex{\underline{W G_1^* A_1 \Im G_2^t A_2}} \\
    &+\frac{m_1}{N}\sum_\mu  \Ex{ N \dg\tilde{S}^{\circ}_{\mu} G_1} \Ex{\Im G_1 A_1 \Im G_2^t A_2 N\dg(\tilde{S}_\mu)} \\ & + \frac{m_1}{N} \sum_{\mu}  \Ex{ N \dg\tilde{S}^{\circ}_{\mu} \Im G_1} \Ex{G_1^* A_1 \Im G_2^t A_2 N\dg(\tilde{S}_\mu)}\\
    &+ \frac{\Im m_1}{N}\sum_\mu  \Ex{ N \dg\tilde{S}^{\circ}_{\mu} G_1^*} \Ex{ G_1^* A_1 \Im G_2^t A_2 N\dg(\tilde{S}_\mu)}\\
    &+ m_1 \Ex{G_1 - m_1} \Ex{\Im G_1A_1\Im G_2^tA_2}+ m_1 \Ex{\Im G_1 - \Im m_1} \Ex{G_1^* A_1\Im G_2^tA_2} \\
    &+ \Im m_1 \Ex{ G_1^* - \overline{m_1}} \Ex{G_1^* A_1\Im G_2^tA_2}\\
&+\frac{m_1}{N^2}\sum_\mu  \Ex{\Im G_1A_1 G^t_2 N\dg(\tilde{S}_\mu)A_2^t \Im G_2 N\dg(\tilde{S}_\mu)}\\&+\frac{m_1}{N^2}\sum_\mu \Ex{\Im G_1A_1\Im G^t_2 N\dg(\tilde{S}_\mu)A_2^t G_2^* N\dg(\tilde{S}_\mu)}\\
&+\frac{\Im m_1}{N^2}\sum_\mu  \Ex{ G_1^*A_1 G^t_2 N\dg(\tilde{S}_\mu)A_2^t \Im G_2 N\dg(\tilde{S}_\mu)}\\&+\frac{\Im m_1}{N^2}\sum_\mu \Ex{G_1^* A_1\Im G^t_2 N\dg(\tilde{S}_\mu)A_2^t G_2^* N\dg(\tilde{S}_\mu)}.
\end{split}
\end{equation}
Like in the proof of Lemma \ref{GAGtA} we use \cite[(5.34),(5.35)]{cipolloni-erdos-schroder-2021} to get
\begin{equation*}
    \begin{split}
        &\left| \Ex{ N \dg\tilde{S}^{\circ}_{\mu} \Im G_1} \Ex{G_1^* A_1 \Im G_2^t A_2 N\dg(\tilde{S}_\mu)}\right| \prec \frac{\sqrt{\rho_1}}{N\sqrt{\eta_1}}\cdot \rho_2\Lambda_{k+1}^2 \prec \frac{\rho_1\rho_2\Lambda_{k+1}^2}{\sqrt{NL}}\,,\\
    & \left| \Ex{\Im G_1A_1\Im G^t_2 N\dg(\tilde{S}_\mu)A_2^t G_2^* N\dg(\tilde{S}_\mu)}  \right| \prec \frac{\rho_1\rho_2\Lambda_k \Lambda_{k+1}}{\sqrt{\eta_1\eta_2}} \prec \frac{\rho_1\rho_2\Lambda_{k+1}^2}{L},\\
    \end{split}
\end{equation*}

The second term of \eqref{eq:Gtcase-general} is estimated differently from the diagonal case:
\begin{equation}
\begin{aligned}
     |\Im m_1\Ex{A_1(\Im G_2^t - \Im m_2 I) A_2}| &\leq \rho_1 |\Ex{\Im G_2^t - \Im m_2 I} \Ex{A_2A_1}| + \rho_1 |\Ex{\Im G_2^t (A_2A_1)^\circ}|\\
     &\prec \frac{\rho_1 \rho_2}{L} + \rho_1 \rho_2 \left(\frac{ \Lambda_k}{\sqrt{NL}} + \frac{\Lambda_{k+3}}{NL}\right)\,.
\end{aligned}
\end{equation}
The bounds for the other terms in \eqref{eq:Gtcase-general} can be obtained similarly. 
\end{proof}

\subsection{Proof of Main Result}

We now have enough results to prove our main Theorem  \ref{thm:Xicontrolparam}. 

\begin{proof}[Proof of Theorem \ref{thm:Xicontrolparam} ]
The result of Lemma \ref{lem:GAGAgeneral} combined with Lemma \ref{GAGA&Xi} shows that,
\begin{equation}
    \Lambda_k^2 \prec 1 + \frac{\Lambda_{k+4}^2}{\sqrt{J}}.
\end{equation}

Starting from $k=1$ and iterating this bound shows that,
\begin{equation}
    \Lambda_1^2 \prec 1 + \frac{\Lambda_{1+4t}^2}{(\sqrt{J})^t}.
\end{equation}
We can choose $t = \lceil 8/\epsilon \rceil$ and apply the trivial bound $\Lambda_{1+4 t} \prec \frac{1}{\eta}$ as well as $J \ge N^{\epsilon}$ to show that $\Lambda_1^2 \prec 1$.

\end{proof}

\section{Proof of Lemma \ref{lem:underline}}
\subsection{Cumulant expansion}

In this section we use cumulant expansion to estimate the moments
\begin{equation}\label{moment}
    \E |\langle \underline{WG_1B_1G_2B_2 \ldots G_lB_l} \rangle|^{2p}.
\end{equation}
Parts of this section were adapted from the paper \cite{cipolloni-erdos-schroder-2021}.

For simplicity we assume that $B_i \in \mathbb{M}_k$ for all $i \in [N]$. It is easy to see from our proof that if $B_i$ are from different families $\mathbb{M}_k$, each $B$ provides the corresponding $\Lambda$ in the bound.

For any $m, n \in \mathbb{Z}_+$ define a $N\times N$ matrix $\kappa^{m,n}$, such that its entries $\kappa^{m,n}_{ab}$ are the joint cumulants of $m$ copies of $w_{ab}$ and $n$ copies of $w_{ba}$. Note that $\kappa^{1,1} = S$ and $\kappa^{2, 0} = 0$.

We use the following cumulant expansion:
\begin{equation}\label{eq:CumulantExpansion}
    \E w_{ab}f(W) = \sum_{k = 1}^R \sum_{m+n=k} \kappa^{m+1, n}_{ab} \E \partial_{ab}^m \partial_{ba}^n f(W) + \Omega_R,
\end{equation}
where $\partial_{ab} = \partial_{w_{ab}}$.

Applying the expansion \eqref{eq:CumulantExpansion} to \eqref{moment} $2p$ times with respect to each $W$ allows us to express the moments \eqref{moment} in terms of Feynman diagrams (Lemma \ref{lem:cumulant}).  We understand that the following definition is quite long, but we will soon give an example that will make these concepts more concrete.

\begin{definition}
Define the class of diagrams $\mathcal{G}$ as follows. Each diagram $\Gamma$ is a graph with two types of vertices $V = V_\kappa \cup V_i$ that are called $\kappa$-vertices and internal vertices and two types of edges $E = E_\kappa \cup E_g$ called $\kappa$-edges and $G$-edges. For any vertex $v \in V$ its $G$-degree $d_g(v)$ is defined as its degree in the graph $(V, E_g)$. Internal vertices $v \in V_i$ satisfy $d_g(v) = 2$ and $\kappa$-vertices can be partitioned $V_\kappa = \bigcup_{k \ge 2} V_\kappa^k$ according to their degree, i.e. $d_g(v) = k$ for $v \in V_\kappa^k$. $\kappa$-edges can be partitioned $e_\kappa = \bigcup_{k \ge 2} E_\kappa^k$ so that any $e \in E_\kappa^k$ connects two vertices from $V_\kappa^k$.

Each $\kappa$-edge $e = (v, w)$ carries labels $r(e)$, $s(e)$ and the value of $\kappa^{r(e), s(e)}_{vw}$. Each edge $e \in E_\kappa^2$ carries an additional label $h(e) \in \{\text{mat}, \text{res}\}$, which will record whether the edge comes from the derivative $\partial_{e}$ hitting a matrix $W$ or a resolvent $G_k$. Each $G$-edge $e$ has labels $i(e), t(e), *(e) \in \{0,1\}$ recording the type of the resolvent $e$ represents (imaginary part, transpose and adjoint respectively). Label $z(e)$ records the parameter of the resolvent. Labels $L(e)$ and $R(e)$ record deterministic matrices that resolvent is multiplied by.  
\end{definition} 

\begin{remark}
In this paper $L(e)$ and $R(e)$ will be products of matrices $B_k$ and $diag( \tilde{S}^{\circ}_\mu)$ defined in Assumption \ref{asmp:squareroot}.
\end{remark}

In addition to the definition of diagrams, we also need to introduce the notion of values associated to each diagram. On an intuitive level, a diagram represents some product of matrix quantities organized in a particular way. The following definition formalizes the exact quantity associated to each diagram.

\begin{definition}
For each $\Gamma \in \mathcal{G}$ and each $e \in V_i$ define the value $\mathcal{G}^e$ of the edge $e$ as the resolvent $L(e)G(z(e))R(e)$ with imaginary part, transpose and $*$ applied according to the labels $i(e), t(e), *(e)$. For each $e\in E_\kappa$ define its value $\mathcal{G}^e$ as $\kappa^{r(e), s(e)}$.

Define the value of the diagram $\Gamma$ as follows.
\begin{equation}
    \Val(\Gamma) = \sum_{a_{v} \in [N], v \in V} \prod_{\{x,y\} \in E} \mathcal{G}^{\{x,y\}}_{a_x,a_y}
\end{equation}
\end{definition}

Here, we construct a few examples of the diagrams that appear after applying the cumulant expansion.
Let us the consider the following terms,
\begin{equation}\label{eq:example-expansion}
\begin{aligned}
    \E\left|\Ex{\underline{W G_1 B_1 G_2 B_2}}\right|^2 &= \E \sum_{a,b} \kappa^{1,1}_{a,b} (G_1 B_1)_{bc} (G_2 B_2)_{ca} (B_2^* G_2^*)_{ad} (B_1^* G_1^*)_{db} \\
    &+ \E\sum_{a,b,c,d} \kappa^{1,1}_{a,b} \kappa^{1,1}_{c,d} (G_1 B_1 G_2)_{bd} (G_2 B_2)_{ca} (B_2^* G_2^* B_1^* G_1^*)_{db} (G_1^*)_{ac} \\
    &+ \E\sum_{a,b,c,d} \kappa^{2,1}_{a,b} \kappa^{1,1}_{c,d} (G_1)_{bd} (G_1 B_1 G_2)_{ca} (G_2 B_2)_{ba} (B_2^* G_2^*)_{db} (G_2^* B_1^* G_1^*)_{ac} \\
    &+ \ldots
\end{aligned}
\end{equation}

Figure \ref{fig:first-term} shows the diagram corresponding to the first term of \eqref{eq:example-expansion}. Each edge has its value written next to it. On the right of Figure \ref{fig:first-term} we show the edge labels in more detail. Figure \ref{fig:second-third-term} shows the diagrams corresponding to the other two terms of \eqref{eq:example-expansion}.

\begin{figure}
    \centering
        \includegraphics{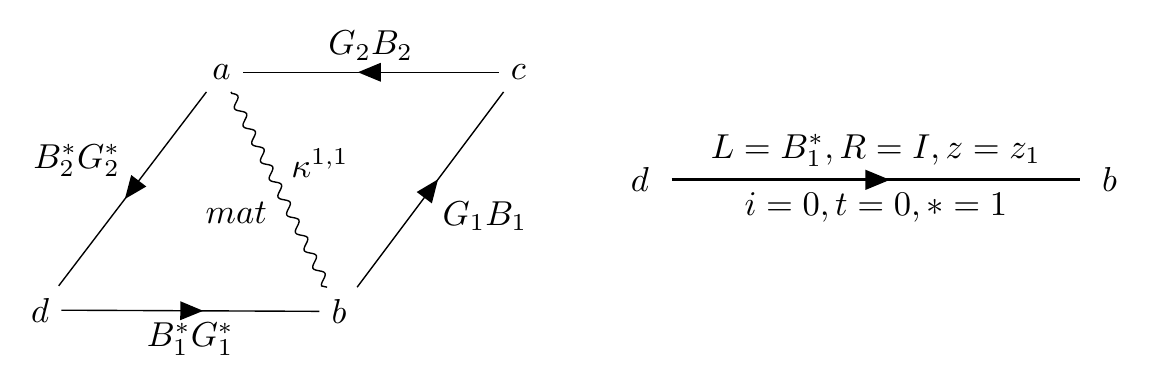}

    \caption{Diagram on the left corresponds to the first term of \eqref{eq:example-expansion}. On the right we show all labels of the edge $\{d,b\}$ with value $B_1^*G_1^*$. }
    \label{fig:first-term}
\end{figure}

\begin{figure}
    \centering
    \includegraphics{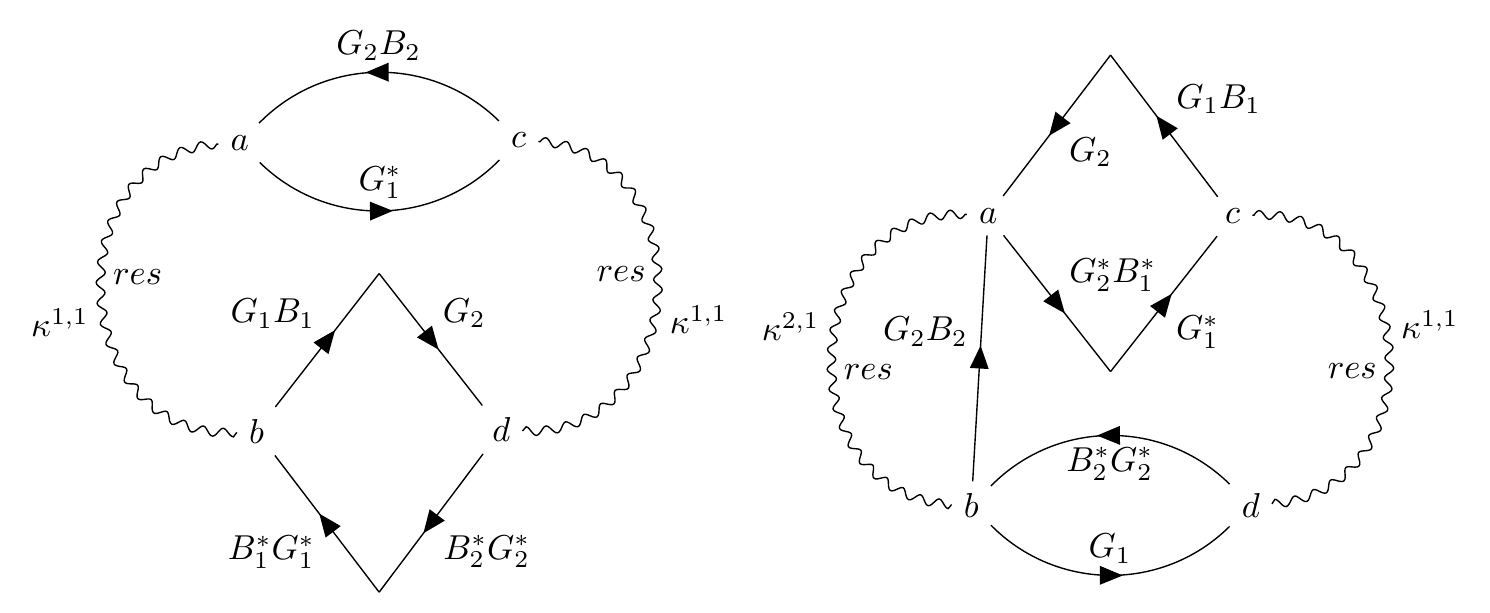}
    
    
 
    

    \caption{The diagrams corresponding to the second and third terms of \eqref{eq:example-expansion}.}
    \label{fig:second-third-term}
\end{figure}


The following definition is similar to the properties (P1)-(P8) in Proposition 5.3 of \cite{cipolloni-erdos-schroder-2021}. The main purpose of the definition is to encompass the properties of the graphs produced by cumulant expansion that are most important for our later counting bound. We remark that most of these properties are mechanical consequences of considering the algorithm of cumulant expansion.

\begin{definition}\label{regular}
A diagram is said to be $(l,p, i,\mathfrak a,\mathfrak t)$-regular if there exist a subset $V_o$ of orthogonality vertices such that the following condition holds:
\begin{enumerate}
    \item The graph $(V_\kappa,E_\kappa)$ is a perfect matching.
    \item The internal vertices satisfies $|V_i| = 2(l-1)p$. The edges satisfy $1 \leq |E_\kappa| \leq 2p$, $\#\{e \in E_g: i(e) =1 \} = 2 ip$, and $|E_g| = \sum_{e \in E_\kappa} d_g(e) + 2(l-1)p \geq 2p$.
    \item For any $\kappa$ edge $(uv) \in E_\kappa^k$, the $G$-degrees of $u,v \in V_\kappa$ satisfy $d^{in}_g(u) = d^{out}_g(v)$, $d^{in}_g(v) = d^{out}_g(u)$, and $d_g(u) = d_g(v) \geq 2$. We can define the $G$-degree of $(uv)$ as $d_g(uv): =d_g(u) = d_g(v) =k$.
    \item Every $E_g$-cycle on $V_\kappa^2 \cup V_i$ contains at least two vertices in $V_\kappa^2$.
    \item Denoting the number of isolated cycles in $(V_\kappa \cup V_i, E_g)$ with at most $k$ vertices in $V_o$ by $n_{cyc}^{o=k}$, we have $2 n_{cyc}^{o=0} + n_{cyc}^{o=1} \leq 2 |E_\kappa^2| - |V_o \cap V_\kappa^2|$.
    \item $|V_i \cap V_o| = 2p(a +t - \mathbf 1\{l \in \mathfrak a \cup  \mathfrak t\})$.
    \item If $l \in \mathfrak a \cup \mathfrak t$, then $2 |E^2_\kappa| + |E^{\geq 3}_\kappa| - 2p \leq |V_o \cap V_\kappa^2| \leq 2p$, otherwise $V_o \cap V_\kappa^2$ is empty.
    \item\label{hist-label-prop} For any $\kappa^{1,1}$ edge $e$, the number of its endpoint in $V_o$ is either $0$ if $h(e) = \text{mat}$, or at most $1$ if $h(e) = \text{res}$.
\end{enumerate}
\end{definition}

With the notion of diagrams in hand, we can now describe the the graph produced by cumulant expansion, which reduces the computations of moments of our renormalization terms to quantities defined on our $(l,p,i,\mathfrak{a},\mathfrak{t})$ regular graphs. 

\begin{lemma}[Cumulant expansion]\label{lem:cumulant}For any $p \in \N$, there is a collection of graphs $\mathcal G^{av}_p$ such that
\begin{align}\label{eq:value-expansion}
    &\E|\Ex{\underline{WG_1B_1...G_lB_l}}|^{2p} = \sum_{\Gamma \in \mathcal G^{av}_p}\E \Val(\Gamma) + O(N^{-2p})\,.
\end{align}
Furthermore, for all diagrams in  $ \mathcal G^{av}_p$ are ``$(l,p, i,\mathfrak a,\mathfrak t)$-regular'' in the sense of Definition \ref{regular}.
\end{lemma}

\begin{proof}
This is the result of Proposition 5.3 in \cite{cipolloni-erdos-schroder-2021}. The only modification we make is the additional property \ref{hist-label-prop} from Definition \ref{regular}. This property holds because the vertices from $V_\kappa^2$ can only be selected as orthogonality vertices if they appear as a result of the derivative acting on a $G$ (see (orth-2) in the proof of Proposition 5.3 in \cite{cipolloni-erdos-schroder-2021}).
\end{proof}

In our proof we need to emphasize another property of the diagrams appearing in equation \eqref{eq:value-expansion}.

\begin{lemma}\label{lem:pureG}
For any $\Gamma \in \mathcal{G}_p^{av}$ and any $\kappa^{1,1}$-edge $e$ with $h(e) = \text{res}$, one of its endpoint has one outgoing $G$-edge $e'$ with $L(e') = I$ and one incoming edge $e''$ with $R(e'') = I$.
\end{lemma}

\begin{proof}
Suppose $e = (ab)$ comes from the cumulant expansion with respect to $w_{ab}$ in $w_{ab} \langle  \Delta^{ab} G_1 B_1 \ldots G_l B_l\rangle$. Since $h(e) = \text{res}$, $\partial_{ba}$ hits a resolvent $G_k$, which becomes $G_k \Delta^{ba} G_k$. Then vertex $b$ has an outgoing edge $e'$ with resolvent $G_1$ and $L(e') = I$ and an incoming edge with resolvent $G_k$ and $R(e) = I$.
\end{proof}

Our final lemma computes the values of regular graphs along with the extra condition that $\kappa_{1,1} = \frac{1}{N}$. This lemma is from \cite{cipolloni-erdos-schroder-2021}. The reason we cannot apply this directly to our cumulant expansion is that the value of $\kappa^{1,1}$ we use is not uniform; our work in the next section is to modify the graphs so that we can reduce  the computation of our graph values to those that appear in the following lemma.

\begin{lemma}\label{lem:value}
If $\Gamma_{\kappa,G}$ is $(l,p, i,\mathfrak a,\mathfrak t)$-regular, $\kappa^{1,1}_{xy} = \frac{1}{N}$, $|\kappa^{p,q}_{xy}| \leq C N^{-(p+q)/2}$, and $L(e),R(e) \in \mathbb M_k \cup \mathbb M_k^\circ$ for all $G$-edges $e$, then
\begin{equation}
    \Val(\Gamma_{\kappa,G}) \prec \begin{cases}
    \rho^{2(b+1)p}N^{2bp}L^{-2bp}, &b = l\,,\\
    \Lambda_k^{2(a+t)p}\rho^{2ip \vee 2(b+1)p
    }N^{p(a+t+2b)}L^{-p(1+2b)},&b<l\,,
    \end{cases}
\end{equation}
where $b:= l - a -t$.
\end{lemma}

\subsection{The Graph Splitting Procedure}

We remark that there is only one difference between the case that we are considering here and the diagrams from \cite{cipolloni-erdos-schroder-2021}: they use the fact that the lowest order cumulants $\kappa^{1,1}$
are all uniformly $\frac{1}{N}$. This allows them to re-express some of the quantities related to the diagrams in terms of traces of matrix products. This is the key step that allows them to apply Proposition 5.6 to bound the value of the graph.  In the absence of this important condition, what we must do is find a way to take our graphs into an expression that would be useful.

We find a procedure that takes any diagram $\Gamma$ and re-expresses it as a sum of other diagrams. Namely,
\begin{equation}
\text{Val}(\Gamma) = \sum_{\mu} \Val(\tilde{\Gamma}^\mu).   
\end{equation}
The details of our transformation will show that the diagrams $\tilde{\Gamma}^\mu$  are $(l,p, i,\mathfrak a,\mathfrak t)$-regular diagrams along with the property that the new `$\kappa^{1,1}$' edges have value $\frac{1}{N}$. Formally, we find another way to write the covariance matrices $\kappa^{1,1}$ and incorporate these terms into one of the $L$ or $R$ matrices that multiply the $G$. The end result of this procedure is to formally treat the old $\kappa^{1,1}$ edges as having value $\frac{1}{N}$. We now begin to describe this procedure more formally.

For every $\kappa^{1,1}$ edge, we will decompose the diagram $\Gamma$ into $N+1$ further diagrams. Thus, if we let $E_2$ be the set of all $\kappa^{1,1}$ edges, we will decompose $\Gamma$ into $(N+1)^{|E_2|}$ edges. The graph splitting procedure essentially treats every edge independently, so to describe the construction, it is best to consider the case that there is only a single $\kappa^{1,1}$ edge.

First, consider a $\kappa^{1,1}$ edge $e=(x,y)$. In case $h(e) = res$, we know from Lemma \ref{lem:pureG} that one of $x$ or $y$ has the property that it has a single incoming edge with $R(e)=I$ and a single outgoing edge with $L(e)=I$. Assume that this vertex is $x$.  When we split $\kappa^{1,1}$ later, then this property allow us to introduce a trace 0 matrix in a location that is useful for cancellations. We remark that by Definition \ref{regular} (1) if there were more $\kappa^{1,1}$ edges, this vertex $x$ would not be shared with the other $\kappa^{1,1}$ edges. In case $h(e) = mat$, $x$ can be chosen arbitrarily among the two vertices adjacent to the $\kappa^{1,1}$ edge.

We have the following matrix decomposition of $S_{xy}:=\kappa^{1,1}_{x,y}$. 
\begin{equation}
    S_{xy} = \sum_{\mu} \tilde{S}_{x \mu} \tilde{S}_{\mu y}.    
\end{equation}
Though this procedure formally introduces the vertex $\mu$, these vertices will not be introduced into our diagram. Instead, the first $N$ of $N+1$ diagrams will be indexed by this vertex $\mu$; we call these diagrams $\tilde{\Gamma}^1,\ldots,\tilde{\Gamma}^N$. The remaining graph we will call $\tilde{\Gamma}^{\ext}$, and will be derived via a more complicated resummation procedure. 

After fixing a value of $\mu$, we still have to perform more manipulation in order to derive the diagrams $\tilde{\Gamma}^\mu$. This procedure is quite similar to those performed in Sections \ref{sec:diagM} and \ref{sec:generalM}; namely, we further split  $\tilde{S}_{x\mu}$ and interpret this as a trace $0$ matrix. Recall the traceless part as follows:
\begin{equation} \label{eq:trace}
\begin{aligned}
  & \tilde{S}^{\circ}_{x\mu}= \tilde{S}_{x\mu} - \frac{1}{N}\\
    &\dg\tilde{S}^\circ_\mu = \dg\tilde{S}_\mu -\Ex{\dg\tilde{S}_\mu} = \dg\tilde{S}_\mu - \frac{1}{N}.
\end{aligned}
\end{equation}

We define the diagram $\tilde{\Gamma}^\mu$ as follows. We take the diagram $\Gamma$ and redefine labels so that for the incoming edge to vertex $x$, we multiply $R$ on the right by $N\dg \tilde{S}^{\circ}_{\mu}$. At $y$, we change the label so that at the incoming edge to $y$, we multiply $R$ on the right by $N\dg \tilde{S}_{\mu}$. Finally, we change the value of $\kappa^{1,1}$ formally to $\frac{1}{N}$ on the edge $\{x, y\}$.

In case $h(e) = res$, the main benefit is that even though we remove the orthogonality at vertex $y$, vertex $x$ becomes a trace zero orthogonality vertex.
Thus, there is no net loss in the number of orthogonality vertices in all of the diagrams $\tilde{\Gamma}^{\mu}$ for $\mu$ between $1$ and $N$. Moreover, the traceless matrix involved in this new orthogonality vertex is in $\mathbb M_0^\circ$.

In case $h(e) = mat$, neither $x$ nor $y$ are orthogonality vertices of $\Gamma$, so our procedure doesn't affect $V_o$.

Finally, we define our final diagram $\tilde{\Gamma}^{\ext}$ by taking the diagram $\Gamma$ and changing the label $\kappa^{1,1}$ to $\frac{1}{N}$ at the edge $\{x,y\}$. It is easy to check that
\begin{equation}
    \Val(\Gamma) = \frac{1}{N}\sum_{\mu = 1}^N \Val(\tilde{\Gamma}^{\mu}) + \Val(\tilde{\Gamma}^{\ext})\,.
\end{equation}

Here is an example of the graph splitting procedure. We only draw the verteces and edges that connect to the $\kappa^{1,1}$ edge.
\begin{center}
    \includegraphics{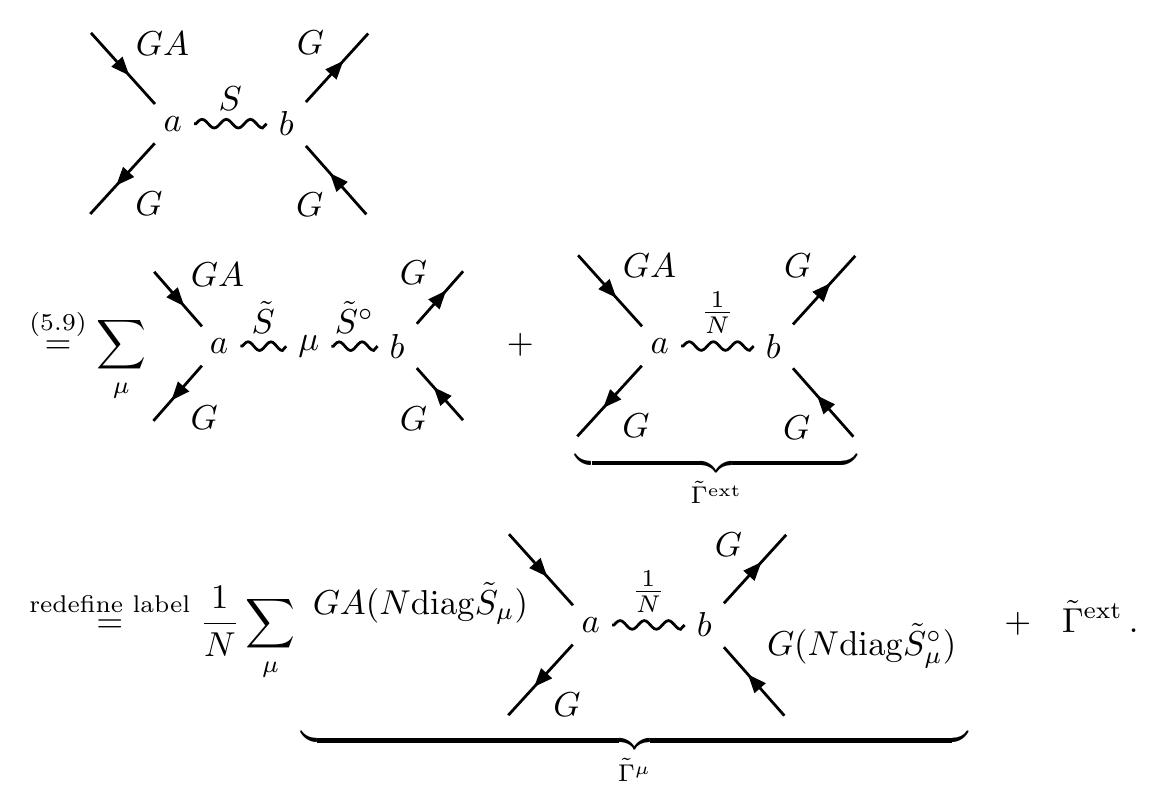}
\end{center}

\begin{lemma}
The diagrams $\tilde{\Gamma}^\mu$ or $\tilde{\Gamma}^{\ext}$ that have been constructed are $(l,p, i,\mathfrak a,\mathfrak t)$-regular diagrams. 

\end{lemma}
\begin{proof}
Let $e=(x,y)$ be the single $\kappa^{1,1}$ edge in the original diagram $\Gamma$.

The only modifications that we have to the diagrams $\tilde{\Gamma}^i$ from the original diagram $\Gamma$ are in the labels of the edges are are adjacent to the vertices $(x,y)$. In addition, vertices $x$ and $y$ are our only possible changes for $V_o$. From this property, it is clear to see that properties 1-4 of Definition \ref{regular} still hold after the modification. Furthermore, property 5 of Definition \ref{regular} is a consequence of property 4 regardless of the choice of $V_o$. This is a simple counting argument.

With regards to the sets of vertices $V_o$, we describe a bit more elaborately what changes are made. Due to property $8$ of Definition \ref{regular}, we know that at most one of $x$ and $y$ is in the set $V_o$.
Recall that we have earlier chosen the vertex $x$ to lie between the product of two pure $G$ matrices without any intermittent matrix product. Once we multiply in between the two pure $G$ matrices at $x$ by the matrix $\text{diag}(\tilde{S}^{\circ})_{\mu}$, we see that $x$ is in $V^{0tr}$ of $\tilde \Gamma^i$. Regardless of whether $x \in V_o$ or $y \in V_o$ in the graph $\Gamma$, we put the vertex $x$ into $V_o$ for $\tilde \Gamma^i$. Thus, the cardinality of any the sets $V_o$ and $V_o \cap V_\kappa^2$ do not change. Thus,  properties 6-8 of Definition \ref{regular} still hold. 
\end{proof}

The previous discussion has established the following lemma.
\begin{lemma}\label{lem:finalsplit}
Every $\Gamma \in \mathcal G_p$ can be split into a set of diagrams, $\widetilde{\Gamma}^{(\mu_1,...,\mu_{m})}$ for $\mu_i \in [n] \cup \{0\},1 \leq i \leq m$ with $m$ being the number of $\kappa^{1,1}$ edges in $\Gamma$, in the sense that
\begin{equation}
    \Val(\Gamma_{\kappa,G}) = \sum_{\mu_i \in [n] \cup \{\ext\},1 \leq i \leq m} \Big(\frac{1}{N} \Big)^{\# \{i: \mu_i \not = \ext\}} \Val(\widetilde{\Gamma}^{(\mu_1,...,\mu_{m})}_{\kappa',G})\,,
\end{equation}
where ${\kappa'}^{1,1} = \frac{1}{N}$ and ${\kappa'}^{p,q}= {\kappa}^{p,q}$ for $(p,q) \not= (1,1)$.
Furthermore, all of the diagrams $\widetilde{\Gamma}^{(\mu_1,...,\mu_{m})}_{\kappa',G}$ are $(l,p, i,\mathfrak a,\mathfrak t)$-regular.
\end{lemma}

\begin{proof}[Proof of Lemma \ref{lem:underline}]
The high probability estimates on $\Ex{\underline{WG_1 A_1 G_2 A_2}}$ and other similar quantities are readily derived by computing high moments and applying Markov's inequality. The cumulant expansion Lemma \ref{lem:cumulant} gives an expression of the moments in terms of values of graphs. The values of these graphs are determined by the splitting procedure in Lemma \ref{lem:finalsplit} and the evaluation of graph values in appropriate conditions as in Lemma \ref{lem:value}. The combinatorics of the sum over graphs is exactly the same as that of \cite{cipolloni-erdos-schroder-2021} and we can derive the exact same estimates.
\end{proof}




\newpage

\bibliographystyle{abbrv}
\bibliography{bibliography}

\begin{thebibliography}{10}

\bibitem{Benigni2020}
L.~Benigni.
\newblock {Eigenvectors distribution and quantum unique ergodicity for deformed
  Wigner matrices}.
\newblock {\em Annales de l'Institut Henri Poincaré, Probabilités et
  Statistiques}, 56(4):2822 -- 2867, 2020.

\bibitem{benigni2022fluctuations}
L.~Benigni and G.~Cipolloni.
\newblock Fluctuations of eigenvector overlaps and the berry conjecture for
  wigner matrices.
\newblock {\em arXiv preprint arXiv:2212.10694}, 2022.

\bibitem{LP2021}
L.~Benigni and P.~Lopatto.
\newblock Fluctuations in local quantum unique ergodicity for generalized
  wigner matrices.
\newblock {\em arXiv:2103.12013}, 2021.

\bibitem{isotropic}
A.~Bloemendal, L.~Erd{\H o}s, A.~Knowles, H.-T. Yau, and J.~Yin.
\newblock Isotropic local laws for sample covariance and generalized {W}igner
  matrices.
\newblock {\em Electron. J. Probab.}, 19(33):1--53, 2014.

\bibitem{BGS}
O.~Bohigas, M.~Giannoni, and C.~Schmid.
\newblock Characterization of chaotic quantum spectra and universality of level
  fluctuation laws.
\newblock {\em Phys. Rev. Lett.}, 52(1-4), 1984.

\bibitem{BouErdYauYin2015}
P.~Bourgade, L.~Erd{\H o}s, H.-T. Yau, and J.~Yin.
\newblock Fixed energy universality for generalized {W}igner matrices.
\newblock {\em Communications on Pure and Applied Mathematics},
  69(10):1815--1881.

\bibitem{BouErdYauYin2017}
P.~Bourgade, L.~Erd{\H o}s, H.-T. Yau, and J.~Yin.
\newblock Universality for a class of random band matrices.
\newblock {\em Advances in Theoretical and Mathematical Physics},
  21(3):739--800, 2017.

\bibitem{BouYau2017}
P.~Bourgade and H.-T. Yau.
\newblock The eigenvector moment flow and local quantum unique ergodicity.
\newblock {\em Communications in Mathematical Physics}, 350(1):231--278, 2017.

\bibitem{PartI}
P.~Bourgade, H.-T. Yau, and J.~Yin.
\newblock Random band matrices in the delocalized phase, {I}: Quantum unique
  ergodicity and universality.
\newblock {\em Communications on Pure and Applied Mathematics},
  73(7):1526--1596, 2020.

\bibitem{BreHik96}
E.~Br´ezin and S.~Hikami.
\newblock Correlations of nearby levles induced by a random potential.
\newblock {\em Nucl. Phys. B}, 476:697--706, 1996.

\bibitem{BreHik97}
E.~Br´ezin and S.~Hikami.
\newblock Spectral form factor in a random matrix theory.
\newblock {\em Phys. Rev. E}, 55:4067--4083, 1997.

\bibitem{cipolloni2023gaussian}
G.~Cipolloni, L.~Erd{\H{o}}s, J.~Henheik, and O.~Kolupaiev.
\newblock Gaussian fluctuations in the equipartition principle for wigner
  matrices.
\newblock {\em arXiv preprint arXiv:2301.05181}, 2023.

\bibitem{cipolloni2023optimal}
G.~Cipolloni, L.~Erd{\H{o}}s, J.~Henheik, and D.~Schr{\"o}der.
\newblock Optimal lower bound on eigenvector overlaps for non-hermitian random
  matrices.
\newblock {\em arXiv preprint arXiv:2301.03549}, 2023.

\bibitem{cipolloni-erdos-schroder-2021}
G.~Cipolloni, L.~Erd{\H{o}}s, and D.~Schr{\"o}der.
\newblock Eigenstate thermalization hypothesis for wigner matrices.
\newblock {\em Communications in Mathematical Physics}, 388(2):1005--1048,
  2021.

\bibitem{cipolloni2022optimal}
G.~Cipolloni, L.~Erd{\H{o}}s, and D.~Schr{\"o}der.
\newblock Optimal multi-resolvent local laws for wigner matrices.
\newblock {\em Electronic Journal of Probability}, 27:1--38, 2022.

\bibitem{cipolloni2022rank}
G.~Cipolloni, L.~Erd{\H{o}}s, and D.~Schr{\"o}der.
\newblock Rank-uniform local law for wigner matrices.
\newblock In {\em Forum of Mathematics, Sigma}, volume~10, page e96. Cambridge
  University Press, 2022.

\bibitem{cipolloni-erdos-schroder-2022-normality}
G.~Cipolloni, L.~Erdős, and D.~Schr{\"o}der.
\newblock {Normal fluctuation in quantum ergodicity for Wigner matrices}.
\newblock {\em The Annals of Probability}, 50(3):984 -- 1012, 2022.

\bibitem{deutsch91}
J.~Deutsch.
\newblock Quantum statistical mechanics in a closed system.
\newblock {\em Phys. Rev. A}, 43:2046--2049, 1991.

\bibitem{deutsch18}
J.~Deutsch.
\newblock Eigenstate thermalization hypothesis.
\newblock {\em Rep. Prog. Phys.}, 81, 2018.

\bibitem{ErdKruScho19}
L.~Erd{\H{o}}s, T.~Kruger, and D.~Schroder.
\newblock Random matrices with slow correlation decay.
\newblock {\em Forum Math Sigma}, 7, 2019.

\bibitem{ErdKruSchro20}
L.~Erd{\H{o}}s, T.~Kruger, and D.~Schroder.
\newblock Cusp universality for random matrices i: local law and the complex
  hermitian case.
\newblock {\em Comm. Math Phys}, 378, 2020.

\bibitem{ErdPecRamSchYau2010}
L.~Erd{\H{o}}s, S.~P{\'e}ch{\'e}, J.~A. Ram{\'{\i}}rez, B.~Schlein, and H.-T.
  Yau.
\newblock Bulk universality for {W}igner matrices.
\newblock {\em Comm. Pure Appl. Math.}, 63(7):895--925.

\bibitem{ErdSchYau2011}
L.~Erd{\H{o}}s, B.~Schlein, and H.-T. Yau.
\newblock Universality of random matrices and local relaxation flow.
\newblock {\em Invent. Math.}, 185(1):75-- 119.

\bibitem{ESY_local}
L.~Erd{\H{o}}s, B.~Schlein, and H.-T. Yau.
\newblock Local semicircle law and complete delocalization for {W}igner random
  matrices.
\newblock {\em Commun. Math. Phys.}, 287(2):641--655, 2008.

\bibitem{ErdYau2012singlegap}
L.~Erd{\H{o}}s and H.-T. Yau.
\newblock Gap universality of generalized wigner and beta ensembles.
\newblock {\em J. Eur. Math. Soc.}, 17:1927-- 2036.

\bibitem{ErdYau2012}
L.~Erd{\H{o}}s and H.-T. Yau.
\newblock Universality of local spectral statistics of random matrices.
\newblock {\em Bull. Amer. Math. Soc. (N.S.)}, 49(3):377-- 414.

\bibitem{ErdYauYin2012Univ}
L.~Erd{\H{o}}s, H.-T. Yau, and J.~Yin.
\newblock Bulk universality for generalized {W}igner matrices.
\newblock {\em Probab. Theory Related Fields}, 154(1-2):341-- 407, 2012.

\bibitem{ErdYauYin2012Rig}
L.~Erd{\H{o}}s, H.-T. Yau, and J.~Yin.
\newblock Rigidity of eigenvalues of generalized {W}igner matrices.
\newblock {\em Adv. Math.}, 229(3):1435-- 1515, 2012.

\bibitem{cipolloni-erdos-schroder-2020-functionalCLT}
L.~Erdős, G.~Cipolloni, and D.~Schröder.
\newblock Functional central limit theorems for wigner matrices.
\newblock 12 2020.

\bibitem{KnoYin2013}
A.~Knowles and J.~Yin.
\newblock Eigenvector distribution of wigner matrices.
\newblock {\em Probab. Theory Related Fields}, 155(3):543--582.

\bibitem{KY_isotropic}
A.~Knowles and J.~Yin.
\newblock The isotropic semicircle law and deformation of {W}igner matrices.
\newblock {\em Comm. Pure Appl. Math.}, 66:1663--1749.

\bibitem{Marcinek_thesis}
J.~Marcinek and H.-T. Yau.
\newblock High dimensional normality of noisy eigenvectors.
\newblock {\em arXiv:2005.08425}, 2020.

\bibitem{RudSar1994}
Z.~Rudnick and P.~Sarnak.
\newblock The behaviour of eigenstates of arithmetic hyperbolic manifolds.
\newblock {\em Comm. Math. Phys.}, 161(1):195--213.

\bibitem{srednicki94}
M.~Srednicki.
\newblock Chaos and quantum thermalization.
\newblock {\em Phys. Rev. E}, 50:888--901, 1994.

\bibitem{TaoVu2011}
T.~Tao and V.~Vu.
\newblock Random matrices: universality of local eigenvalue statistics.
\newblock {\em Acta Math.}, 206(1):127--204, 2011.

\bibitem{Wigner}
E.~P. Wigner.
\newblock Characteristic vectors of bordered matrices with infinite dimensions.
\newblock {\em Annals of Mathematics}, 62(3):548--564, 1955.

\bibitem{XuYangYauYin}
C.~Xu, F.~Yang, H.-T. Yau, and J.~Yin.
\newblock Bulk universality and quantum unique ergodicity for random band
  matrices in high dimensions.
\newblock {\em arXiv:2207.14533}.

\end{thebibliography}

\end{document}